\documentclass[11pt,reqno]{amsart}
\usepackage{amsmath,amssymb,amsthm}
\usepackage{fullpage}
\usepackage{mathrsfs,bm,color}
\usepackage{hyperref}
\usepackage{bbm}
\usepackage{prettyref}
\usepackage{xcolor}
\usepackage{appendix}
\hypersetup{
    colorlinks,
    linkcolor={red!50!black},
    citecolor={green!50!black},
    urlcolor={blue!80!black}
}

\numberwithin{equation}{section}
\numberwithin{figure}{section}
 \theoremstyle{definition}
 \newtheorem*{defin*}{\protect\definitionname}
 \newtheorem{exam}{Example}
\theoremstyle{plain}
\newtheorem{theo}{\protect\theoremname}
\newtheorem*{theo*}{\protect\theoremname}
  \theoremstyle{plain}
  \newtheorem{lem}[theo]{\protect\lemmaname}

  \newtheorem{cor}[theo]{\protect\corname}
  \theoremstyle{plain}
  \newtheorem*{remark}{Remark}

  \providecommand{\definitionname}{Definition}
  \providecommand{\lemmaname}{Lemma}
  \providecommand{\theoremname}{Theorem}
  \providecommand{\corname}{Corollary}

\newcommand{\E}{\mathbb{E}}	

\newcommand{\cR}{\mathcal{R}}
\newcommand{\cG}{\mathcal{G}}

\newcommand{\sD}{\mathscr{D}}

\newcommand{\sS}{\mathscr{S}}

\newcommand{\sP}{\mathscr{P}}

\newcommand{\N}{\mathbb{N}}

\renewcommand{\P}{\mathbb{P}}
\newcommand{\R}{\mathbb{R}}

\newcommand{\tr}{\operatorname{tr}}
\newcommand{\iid}{\text{i.i.d. }}
\newcommand{\tensor}{\otimes}
\newcommand{\eps}{\epsilon}

\DeclareMathOperator{\1}{\mathbbm{1}}
\DeclareMathOperator{\Cov}{\operatorname{Cov}}
\newcommand{\eqdist}{\stackrel{(d)}{=}}

\DeclareMathOperator{\prob}{\mathbb{P}}

\newcommand{\abs}[1]{\left\lvert#1\right\rvert}
\newcommand{\norm}[1]{\lvert\lvert#1\rvert\rvert}

\newcommand{\bone}{\1}

\renewcommand{\phi}{\varphi}
\renewcommand{\epsilon}{\varepsilon}

\def\mqcut{{\sf MaxCut }_\kappa}
\def\ER{Erd\H{o}s-R\'enyi }

\def\ve{\varepsilon}
\def\de{{\rm d}}

\title[Inhomogeneous Potts Spin Glass and MAX $\kappa$-CUT]
{A connection between MAX $\kappa$-CUT and the Inhomogeneous Potts Spin glass in the
large degree limit}

\author{Aukosh Jagannath}
\address[Aukosh Jagannath]{Department of Mathematics, University of Toronto}
\email{aukosh@math.toronto.edu}

\author{Justin Ko}
\address[Justin Ko]{Department of Mathematics, University of Toronto}
\email{jko@math.toronto.edu}

\author{Subhabrata Sen}
\address[Subhabrata Sen]{Department of Statistics, Stanford University, California}
\email{ssen90@stanford.edu}

\date{\today}

\begin{document}

\begin{abstract}
We study the asymptotic behavior of the Max $\kappa$-cut on a family of sparse, inhomogeneous random graphs. In the large degree limit, the leading term is a variational problem, involving the ground state of a constrained inhomogeneous Potts spin glass. We derive a Parisi type formula for the free energy of this model, with possible constraints on the proportions, and derive the limiting ground state energy by a suitable zero temperature limit. 
\end{abstract}

\maketitle

\section{Introduction}
Networks arise in various applications in economics, engineering and social sciences. In a typical social science application, the vertices of the network represent individuals, while their relationships are represented by the edges. The study of structural properties of these networks, and algorithms to find these structures are extremely  important in this context. Various random graph models have been introduced to study such real-life networks (see, e.g., \cite{holland83sbm})--- and questions about networks translate directly into questions about random graphs under this approach. 
Graph partition problems are natural class of algorithmic questions which arise in this context. In these problems, the goal is to 
find a partition of the vertex set that maximizes some objective function, 
typically given by a function of the edges.  Graph partition problems are of interest in applications as diverse as community detection  \cite{decelle2011sbm} and VLSI design \cite{kahng2011partition}. 
In this paper, we focus on the Max $\kappa$-cut, an important example in this class. 

For any graph $G= (V,E)$, the Max $\kappa$-cut problem (henceforth denoted as $\mqcut$) 
seeks to divide the vertices, $V$, into $\kappa$ (not necessarily equal) parts such that 
the number of edges between distinct parts is maximized. 
For $\kappa=2$ this reduces to the well known ${\sf{MAXCUT}}$ problem 
(see \cite{poljaktuza1993maxcut} for a survey of the $\sf{MAXCUT}$ problem). 
From the point of view of complexity theory, these questions are usually NP hard in the worst
case. This motivates a study of average case complexity, often formalized by studying this problem on random graph instances. As a first attempt, one seeks to determine the typical behavior of these quantities on a random graph--- this provides a valuable benchmark for comparing the performance of specific algorithms on random instances. 

Such questions have been studied in classical settings, such as  the \ER random graph and random regular graph ensembles.
The key insight in this setting is a connection between statistical physics and the $\mqcut$ problem, enunciated as follows. 
 Any $\kappa$-cut can be represented by an assignment of spins $\sigma \in [\kappa]^N$ to the vertices of the graph. Further, setting $A= (A_{ij})$ to be the adjacency matrix of the random graph $G$, we have, 
\begin{align}
\frac{\mqcut(G)}{N} = \frac{1}{2N} \max_{\sigma \in [\kappa]^N} \sum_{i, j =1 }^N A_{i,j} \bone(\sigma_i \neq \sigma_j).  \label{eq:representation}
\end{align}
In statistical physics parlance, \eqref{eq:representation} establishes a direct relation between the $\mqcut$ and the ground state of the antiferromagnetic Potts model on the graph. Connections between
graph partition problems and statistical physics are, by now, classical  \cite{fu-anderson}. For a
textbook  introduction to the physical perspective on these questions, we refer the reader to \cite{mezard2009information,MPV}.

Physicists predict that the antiferromagnetic nature of the Max $\kappa$-cut should force the quantity to behave as the ground state of a disordered spin glass, and its behavior in graphs with large degrees should be well approximated by properties of ground states in mean field spin glasses.
For the ${\sf {MAXCUT} }$ problem on sparse \ER and random regular graphs, this idea was partially formalized in \cite{dembo2016extremal} and \cite{sen2016optimization}. The authors of \cite{dembo2016extremal}  deduced that for $G\sim G(N, \frac{c}{N})$, as $N \to \infty$, we have, 
\begin{align}
\frac{{\sf MAXCUT}(G)}{N} = \frac{c}{4} + {\sf{P}_*} \sqrt{\frac{c}{4}} + o_c(\sqrt{c}).  \label{eq:prev_result}
\end{align}
Here, ${{\sf P}}_*$ is the limiting ground state energy of the Sherrington-Kirkpatrick model \cite{P3}. 
Here and henceforth in the paper, we say that a sequence of random variables, $(X_N)$, satisfies $X_N = o_c(\sqrt{c})$ if and only if there is a deterministic function $g(c) = o(\sqrt{c})$ such that $\P[|X_N| \leq g(c)] \to 1$ as $N \to \infty$. The first term in the right hand side comes from the standard observation that a typical partition of the vertices will contribute $Nc/4$ edges to the cut in expectation. The second term is the leading order correction, and specifies the difference in size between a typical cut and the MAXCUT. An analogous formula for the $\mqcut$ on sparse \ER and random regular graphs was derived in \cite{sen2016optimization}.

 For practical applications, it is thus of natural interest to determine the typical value of these quantities on 
random graph ensembles that capture natural properties of realistic networks. In practice, networks are typically observed to be sparse and ``inhomogeneous" \cite{albert02review,doro2002survey}.
The simplest random graph models, such as \ER and random regular graphs, lead to instances where the degree distributions are relatively concentrated--- a feature seldom observed in real networks. To address this issue, a plethora of models have been introduced, which faithfully capture some of the observed characteristics of real networks. 
In this paper, we seek to establish formulae similar to \eqref{eq:prev_result} for a general family of graph models using the framework of \cite{dembo2016extremal} and \cite{sen2016optimization}.  Our approach leads naturally to the study of an inhomogeneous Potts spin glass model, which has yet to be studied rigorously in the mathematical literature.

Let us first explain the class of random graph models that we study. 
Our framework will be similar to the one introduced by Soderberg \cite{soderberg2002inhomogeneous} and adopted by Bollobas, Janson and Riordan \cite{BJR2007inhomogeneous}. Furthermore, this model has natural connections to the theory of graphons for dense sequences of random graphs \cite{BC1,BC2}. 
Consider a symmetric 
 kernel $K:[0,1]^2\to[0,\infty)$. We will assume throughout that $K \in L^1([0,1]^2,\de x)$. 
 Given such a kernel, consider the following model for a sequence of inhomogeneous random graphs $\cG_N=(V_N,E_N)$. 
For all $N\geq 1$, we let the vertex set be $V_N =[N]$. The edges will then be added independently with probability 
\begin{align}
\prob[\{i,j\}\in E_N] = \min\Big\{c\frac{\tilde K(i,j)}{N},1 \Big\}, \label{eq:connection}
\end{align}
where $\tilde K$ is the average of $K$ within blocks,
\begin{align}
\tilde{K}_N(i,j) = N^2 \int_{ [\frac{i-1}{N}, \frac{i}{N}] \times [\frac{j-1}{N}, \frac{j}{N}] } K(x,y)\, \de x \de y.  
\end{align}
This specifies the random graph model. The parameter $c$ controls the degree of the vertices. 
We note that this model is more restricted compared to that of Bollobas, Janson and Riordan \cite{BJR2007inhomogeneous}. In the notation of \cite{BJR2007inhomogeneous}, we restrict ourselves to the case where the ground space $\mathcal{S} = [0,1]$ and the measure $\mu$ is the Lebesgue measure. Further, the model introduced in \cite{BJR2007inhomogeneous} is governed by the value of the kernel $K$ on a set of measure zero. Here we average over small partitions of the kernel, so that
we may avoid technical subtleties on sets of measure zero. For reasonable kernels, such as continuous ones, this distinction will be negligible. 

To state our main result in a concrete setting, let us first work in the case when $K$ is block constant. 
That is, we assume that  there are numbers 
\[
0= t_0 < t_1 < t_2 < \ldots < t_{M-1} < 1= t_M
\]
such that $K$ is constant on each square of the form $[t_{j-1}, t_j] \times [ t_{k-1}, t_k]$ for $0\leq j, k \leq M$. Further, we set $\rho^{s} = t_{s}- t_{s-1}$, $s= 1, \ldots, M$. For any such block kernel $K$, let $\mathbf{K}$ denote the $M\times M$ matrix of the values of the kernel on the blocks.  For technical reasons, we will work with block kernels such that the matrix $\mathbf{K}$ is positive definite. Finally, we note that any block constant kernel with finitely many blocks is almost surely bounded. By a standard application of the Efron-Stein inequality \cite{BLM2013conc}, it suffices to study the asymptotic behavior of $\E[\mqcut(G_N)]/N$. 

To analyze this quantity, we introduce the following notation. For any finite set $\sS$, let $\sD$ be the space of proportions, given by 
\begin{equation}
\sD = \Big\{ (d_1^s, \dots, d_{\kappa}^s)_{s \in \sS} \mathrel{}\Big| d_k^{s} \geq 0, \mathrel{} \sum_{k = 1}^{\kappa} d_k^s = 1 ~\forall s \in \sS \Big\}.\label{eq:props}
\end{equation}
In our setting, $\sS=[M]$. Any $d \in \sD$ can be expressed as $d = (d^s)_{s \in \sS}$, where $(d^s)_{s \in \sS}$ is a collection of probability measures on $[\kappa]$. The distribution $d^s$ governs the proportion of vertices in block $s$ which belong to the partition $i$, $1\leq i \leq \kappa$. We will refer to the elements $d \in \sD$ as proportions. The following theorem characterizes the value of the $\mqcut$ problem for inhomogeneous graphs $G_N$ with block constant kernels and large degrees, up to corrections which are $o(\sqrt{c})$. 

\begin{theo}
\label{thm:discretization}
We have, as $N\to \infty$, 
\begin{align}
\lim_{N \to \infty} \E\Big[ \frac{\mqcut(G_N)}{N} \Big] 
&= \sup_{d \in \sD} \Big[ \frac{c}{2} \sum_{s,t = 1}^M \mathbf{K}(s, t) \rho^{s} \rho^{t} \bigl(1- \langle d^s, d^t \rangle \bigr) + \frac{\sqrt{c}}{2} \mathcal{P}(d) \Big] + o(\sqrt{c}). \nonumber
\end{align}
\end{theo}

\begin{remark}
We take this opportunity to comment on the positive definite assumption on the matrix $\mathbf{K}$. One prominent example where $\mathbf{K}$ is not positive definite is the random bipartite graph, where the kernel consists of two off-diagonal blocks. However, note that for $\kappa=2$, the behavior of the ${\sf{MAXCUT}}$ on this graph is trivial, and very different from that established in Theorem \ref{thm:discretization}. 
\end{remark}
\noindent
Note that the leading term in Theorem \ref{thm:discretization} is a variational problem involving the empirical distribution of spins within each block. This variational problem has two terms: the first term governs the expected cut-size, while the second term, of order $\sqrt{c}$, governs the extra contribution which is attained by optimization. It remains to introduce  $\mathcal{P}(d)$. 
It turns out that $\mathcal{P}(d)$ is the limiting ground state energy of the 
 inhomogeneous Potts spin glass model, subject to constraints on the composition of spins within each block. We introduce this model in the rest of the section, and define the constant $\mathcal{P}(d)$ rigorously using a Parisi type formula for the limiting free energy.

\subsection{The Inhomogeneous Potts Model}

We consider a natural 
generalization of the Potts spin glass model that allows for inhomogeneous coupling interactions
between species. The configuration space for this model is  $\Sigma_N=[\kappa]^N$ for some $\kappa\geq2$. 
Let $\sS$ be the finite set in \eqref{eq:props}, each element of which is called a species. 
For each $N$, we are given a partition of $[N]$ indexed by the species as 
\[
[N]=\cup_{s\in\sS} I_s.
\]
We say that $i$ belongs to species $s$ if $i\in I_s$. Conversely, we denote by $s(i)$ the species to which $i$ belongs. 
Let $N_s = |I_s|$. Naturally, this quantity varies in $N$. To obtain a reasonable 
limiting structure, we assume that the proportions converge:
\begin{equation}\label{eq:proportion-def}
\rho_N^s = \frac{N_s}{N}\to\rho^s\in(0,1).
\end{equation}
 The Hamiltonian for this model, $H_N$, is the centered Gaussian process
 \begin{equation}\label{eq:ham-def}
 H_N(\sigma)=\frac{1}{\sqrt{N}} \sum_{i,j=1}^{N} g_{i,j} \1(\sigma_i = \sigma_j), 
 \end{equation}
 where $g_{i,j}$ are independent, centered Gaussian random variables with covariance 
 \begin{equation}
 \E g_{i,j}^2 = \Delta_{s,t}^2, \qquad s,t\in\sS, i\in I_s, j\in I_t.
 \end{equation}
We assume, following \cite{Barra2015,P1},  that the matrix
$\Delta := \Delta_{s,t}^2$ is symmetric and positive definite in $s$ and $t$.
Observe that if we define, for $\sigma^1,\sigma^2\in\Sigma_N$,
\begin{equation}
R_{1,2}^{s}(k, {k'})=\frac{1}{N_s} \sum_{i\in I_s} \1 (\sigma_i^1=k)\1 (\sigma_i^2=k')
\end{equation}
and define the $\kappa\times\kappa$ species overlap matrix
\begin{equation}\label{eq:overlap-def}
R_{1,2}^s=\bigl(R_{1,2}^s (k, {k'}) \bigr)_{k,k'\leq\kappa}
\end{equation}
then  $H_N$ has covariance
\begin{equation}
\Cov(\sigma^1,\sigma^2)= N\sum_{s,t\in\sS}\Delta_{s,t}^2\rho^s_N\rho_N^t(R_{1,2}^s,R_{1,2}^t)
\end{equation}
where $(\cdot,\cdot)$ denotes the Frobenius (or Hilbert-Schmidt) inner product. 

This model is an inhomogeneous extension of the Potts spin glass model, which has been studied extensively in the physics literature \cite{EldSher83b,EldSherr83a, nishimori1983gauge,caltagirone2012dynamical}, and analyzed rigorously in \cite{P2, PVS}. In particular, we break the symmetry between sites. 
When $\kappa=2$, this type of inhomogeneity was introduced in an equivalent form by Barra, Contucci, Mignone and Tantari 
in \cite{Barra2015} where a Guerra-type \cite{Guerra} upper bound for the free energy was obtained. 
The matching lower bound was obtained by Panchenko in \cite{P1}.

Our goal is  to compute constrained free energies of the type
\begin{equation}\label{eq:const-free-def}
F_N(A) = \frac{1}{N}\E \log \sum_{\sigma\in A} e^{\beta H_N(\sigma)}
\end{equation}
for a specific choices of $A$.  Recall the space of proportions $\sD$ defined in \eqref{eq:props}. 
Given a  $d \in \sD$, we have the associated \emph{constrained state space}
\begin{equation}
\Sigma^\epsilon_N(d) = \Big\{ \sigma \in \Sigma_N \mathrel{}\Big|\mathrel{} \sum_{i \in I_s} \frac{\1(\sigma_i = k)}{N_s} \in [d^s_k - \epsilon, d_k^s + \epsilon] \Big\}. \label{eq:constrainedspace}
\end{equation}
We will use the notation $\Sigma_N(d) := \Sigma_N^0(d)$ to denote the constrained state space where the proportions of spins
within species are exactly equal to the proportion $d$. Let $\sD_N \subset \sD$ be the space of feasible constraints for configuration spaces of $N$ coordinates
\begin{equation}\label{eq:feasible-constraint-space}
\sD_N = \{ d \in \sD \mathrel{}\mid\mathrel{} \Sigma_N(d) \neq \emptyset \}.
\end{equation}
We are particularly interested in  computing constrained free energies as in \prettyref{eq:const-free-def} with 
$A=\Sigma(d)$.
Note that since the space of feasible configurations is at most polynomial growth,
classical concentration arguments show that the free energy of the whole system,
$F_N\bigl(\Sigma_N\bigr)$, is asymptotically given by
the maximum of $F_N\bigl(\Sigma_N(d)\bigr)$ over $\sD$. 

As in the Potts spin glass model, among others, the overlap  \prettyref{eq:overlap-def} will play a key role. 
In particular, we will find that the array of overlaps from i.i.d. draws of configurations from (a perturbation of) the Gibbs measure,
will be determined by a path $\pi:[0,1]\to \Gamma_\kappa^{\abs{\sS}}$. Here $\Gamma_\kappa$ is the 
space of $\kappa \times \kappa$ positive definite matrices and $\pi$ is effectively the family of quantile 
transform of the limiting law of the overlap of two independent copies $\sigma^1,\sigma^2$ from the Gibbs measure.

 We now turn to the main result. To this end, we denote 
the space of \emph{left-continuous monotone functions} on $\Gamma_\kappa$ as 
\begin{equation}
\Pi = \left\{ \pi: [0,1] \to \Gamma_\kappa : \text{$\pi$ is left-continuous, $\pi(x) \leq \pi(x')$ for $x \leq x'$} \right\},
\end{equation}
where $\pi(x) \leq \pi(x')$ means that $ \pi(x') - \pi(x)  \in \Gamma_\kappa$. Similarly, we let
\begin{equation}
\bm\Pi =\left\{ (\pi^s)_{s \in \sS} :[0,1]\to\Gamma_\kappa^{\abs{\sS}} \mathrel{}\big|\mathrel{} \pi^s\in\Pi\right\}.
\end{equation}

We also have the following metric on $\bm \Pi$ 
\begin{equation}
\Delta(\bm \pi, \bm {\tilde \pi}) = \int_0^1 \max_{s \in \sS} \| \pi^s(x) - \tilde\pi^s(x) \|_1 \, \de x. 
\end{equation}

For $d \in \sD$ and $r > 0$ we define the following sequences of parameters. Let $(x_i)_{i=1}^r$ be a strictly increasing sequence of numbers
\begin{equation}\label{eq:RPC6}
0=x_{-1}<x_{0}<\ldots<x_{r}=1.
\end{equation}
For each species, let $(Q^s_i)_{i=1}^r$ be an increasing sequence of $\kappa \times \kappa$ positive semi-definite matrices
\begin{equation}\label{eq:RPC5}
0=Q_{0}^{s}\leq Q_{1}^{s}\leq\ldots\leq Q_{r}^{s}= \operatorname{diag}(d_{1}^{s},\ldots,d_{\kappa}^{s})=D^{s}.
\end{equation}

Given these sequences, for each species we can define the Gaussian vector $(z_p^s) \in \R^{\kappa}$ such that
\begin{equation}
\E z_p^s (z_q^s)^T = 2 \delta_{p,q} \bigg( \sum_{t \in \sS} \Delta_{s,t}^2 \rho^t Q_p^t  - \sum_{t \in \sS} \Delta_{s,t}^2 \rho^t Q_{p-1}^t \bigg).
\end{equation}
 The non-random value $X_0^s$ is defined recursively as 
\[X_r^s = \log \sum_{k \leq \kappa} \exp \bigg( \sum_{1 \leq p \leq r} z_p^s(k) + \lambda^s_{k}  \bigg), \]
\begin{equation}\label{eq:recursion}
X^s_k = \frac{1}{x_k} \log \E_k \exp(x_k X^s_{k+1}), \text{ for }0 \leq k < r,
\end{equation}
where $\E_k$ denotes the expectation with respect to only $z_{k + 1}$. Finally, define the functional
\begin{multline}\label{eq:parisiformula}
\mathscr{P}(r, x, d, ((\lambda^s),  (Q^s))_{s\in \sS}) = \sum_{s \in \sS}   \rho^s X_0^s - \sum_{s \in \sS}\sum_{k \leq \kappa}\rho^s \lambda_k^s d^s_k \\- \frac{1}{2} \sum_{\ell = 0}^{r-1} x_{\ell} \sum_{s,t \in \sS} \Delta^2_{st} \rho^s \rho^t \left( (Q^s_{\ell+1},Q^t_{\ell+1}) - (Q^s_{\ell},Q^t_{\ell}) \right).
\end{multline}
The next result characterizes the limiting free energy in these models. 
\begin{theo}\label{thm:ParisiUnConstrained}
For any $\kappa \geq 2$, set of species $\sS$, and sequences $\rho_N^s\rightarrow \rho^s$ 
we have the following:
\begin{enumerate}
\item For any $d \in \sD$ and $\epsilon_N \to 0$ sufficiently slowly, the limit of the constrained free energy is given by
\begin{equation}\label{eq:ParisiConstrained}
\lim_{N \to \infty} F_N\bigl(\Sigma^{\epsilon_N}_N(d) \bigr) =  \inf_{x, r,(\lambda^s, Q^s)_{s \in \sS}} \sP\bigl(r, x, d,(\lambda^s, Q^s)_{s \in \sS} \bigr).
\end{equation}
\item The limit of the unconstrained free energy is given by
\begin{equation}\label{eq:ParisiUnConstrained}
\lim_{N \to \infty} F_N\bigl(\Sigma_N\bigr) = \sup_{d \in \sD} \, \inf_{x, r, (\lambda^s,  Q^s)_{s \in \sS}} \sP\bigl(r, x, d, (\lambda^s, Q^s)_{s \in \sS} \bigr).
\end{equation}
\end{enumerate}
\end{theo}

Before moving forward, we explain the non-trivial obstacles encountered in the inhomogeneous Potts model. First, the symmetry between sites is broken. As a result, interactions both within the species and between the species must be considered. Secondly, the natural overlap structure are matrices \eqref{eq:overlap-def} and are a priori not necessarily positive definite in the limit.

Each of these issues has been studied in the past in \cite{P1} and \cite{P2}. The synchronization property in \cite{P1} connected the species overlaps with the average of the overlaps over the entire system. Similarly, the synchronization property in \cite{P2} proved the overlaps concentrated on the space of Gram matrices in the limit. Another consequence of this result implied the overlap matrices could also be recovered from the trace of the matrix. At the heart of both of these synchronization arguments were generalized Ghirlanda-Guerra type identities that implied an ultrametric underlying structure of the overlaps \cite{PUltra}. The synchronization combines local and global ultrametric properties forcing a rigid distribution structure.

We prove an analogue of the Ghirlanda-Guerra identities which combines those in the inhomogeneous SK and Potts models. This results in a simultaneous synchronization mechanism of the overlap matrices both within and between species. In our setting, we will be able recover the structure of the overlap matrices $R^s_{\ell, \ell'}$ deterministically from the average of the traces of overlaps
\[
\sum_{s \in \sS} \rho^s \tr \bigl( R^s_{\ell, \ell'} \bigr).
\]
These techniques will reduce the problem to a familiar setting, allowing us to derive a formula for the free energy using the Guerra upper bound \cite{Guerra} and the Aizenman-Sims-Starr scheme \cite{AS2} using the characterization method introduced in \cite{AA} and formalized further in \cite{Pspins,PPF}. The resulting functional order parameter is a vector of monotone paths of $\kappa \times \kappa$ matrices. 

In applications, we will be interested, not only in the free energy, but also the maximum of 
\eqref{eq:ham-def} subject to the same constraints.
 To understand the connection between the two, we recall the classical fact from statistical mechanics 
 that the ground state energy of a system can be obtained as the ``zero temperature limit'' of the free energy of this system--- 
 a limit usually referred to as annealing \cite{kirkpatrick1983optimization}. In this case, we take the beaten path, and define free energies of the form
\begin{equation}
F^\beta_N(A) = \frac{1}{N}\E \log \sum_{\sigma\in A} e^{\beta H_N(\sigma)}, \label{eq:free_energy_beta}
\end{equation}
where $H_N$ is as in \eqref{eq:ham-def} for some fixed $\Delta$. Here $\beta$ is the inverse temperature, and the zero temperature limit corresponds to sending $\beta \to \infty$. 
We note that the free energy \eqref{eq:free_energy_beta} corresponds to the Hamiltonian $H_N(\sigma)$ defined as in \eqref{eq:ham-def}, with respect to $\Delta^\beta=\beta^2 \Delta$ instead of $\Delta$. Therefore, a straightforward modification of Theorem \ref{thm:ParisiUnConstrained} implies that for any $d \in \sD$, 
\begin{equation}
\lim_{N \to \infty} F^\beta_N\bigl(\Sigma^{\epsilon_N}_N(d) \bigr) =  \inf_{x, r,(\lambda^s, Q^s)_{s \in \sS}} \sP_\beta\bigl(r, x, d,(\lambda^s, Q^s)_{s \in \sS} \bigr).
\end{equation}
The functional $\sP_\beta(r, x, d, (\lambda^s, Q^s))_{s \in \sS})$ is identical to \eqref{eq:parisiformula} with $\Delta^\beta$ in place of $\Delta$. As a corollary to the theorem above, we obtain an expression for the limiting constrained ground state energies. 

\begin{cor}
\label{cor:ground_state}
For any configuration $d \in \sD$, we have, as $N\to \infty$, for some sequence $\ve_N \to 0$ sufficiently slowly, 
\begin{align}
\lim_{N \to \infty} \frac{1}{N} \E\Big[ \max_{\sigma \in \Sigma_N^{\ve_N}(d)} H(\sigma) \Big] &= \lim_{\beta \to \infty} \frac{1}{\beta} \inf_{x, r,(\lambda^s, Q^s)_{s \in \sS}} \sP_{\beta} \bigl(r, x, d,(\lambda^s, Q^s)_{s \in \sS} \bigr) 
:= \mathcal{P} (d). \nonumber 
\end{align}
\end{cor}

\noindent Note that $\mathcal{P}(d)$ is the relevant constant for the $\mqcut$ in Theorem \ref{thm:discretization}. 
 Exact variational formulas for ground state energies like $\mathcal{P}(d)$ have been obtained recently in several models \cite{auffchen2017,chensen2017,JagTob17}. However, in this setting this remains an interesting question.

\subsection{Applications}
We return to study of graph partitioning problems in this section and study some examples . 
We first note that given a general kernel $K \in L^1$, to determine the $\mqcut$ up to $o(\sqrt{c})$ corrections, it is enough to restrict ourselves to block-constant kernels. To this end, given a kernel $K$ and $M \geq 1$, we construct the kernel $K_1$ by ``coarsening" the kernel $K$, 
\begin{align}
K_1(x,y) = M^2 \sum_{i,j=1}^{M} \bone \Big(x \in \Big[\frac{i-1}{M}, \frac{i}{M} \Big], y \in \Big[\frac{j-1}{M}, \frac{j}{M} \Big] \Big) \int_{\Big[\frac{i-1}{M}, \frac{i}{M}\Big]\times \Big[\frac{j-1}{M}, \frac{j}{M} \Big]} K(s,t ) \, \de s \de t.  
\end{align}

\noindent Let $\tilde{G}_N$ denote the sequence of graphs formed from the kernel $K_1$ using \eqref{eq:connection}.

Then we have, 
\begin{lem}
\label{lemma:discretization}
For any kernel $K$ and $c$, $0 <\delta <1/2 $, we can choose $M := M(c)$ such that for all $N$ sufficiently large, 
\begin{align}
\Big| \E\Big[ \frac{\mqcut(G_N)}{N} \Big]  - \E \Big[ \frac{\mqcut(\tilde{G}_N)}{N} \Big]  \Big| \leq c^{1/2- \delta}. \nonumber 
\end{align}
\end{lem}

\noindent 
Any kernel $K$ is naturally associated with the integral operator  $T_K (f ) (x) := \int K(x,y) f(y) \de y$. We assume that for $f \in L^\infty([0,1])$, $\iint f(x) f(y ) K(x,y) \de x \de y \geq 0$. We note that in this case, the corresponding discretized kernel $K_1$, described in Lemma \ref{lemma:discretization}, inherits the positive definite character. Therefore, $\mqcut(\tilde{G}_N)$ can be determined by an application of Theorem \ref{thm:ParisiUnConstrained}. We now turn to some examples to which our results apply.

\begin{exam}[Finite species block model]
The first example concerns the simple case when the kernel $K$ has an explicit block structure. This model has been proposed
and studied intensely by S\"{o}derberg \cite{soderberg2002inhomogeneous} and Bollobas, Janson and Riordan \cite{BJR2007inhomogeneous}. These models have also been studied as ``Stochastic Block Models" in Statistics, Machine Learning, and Theoretical Computer Science in connection to the community detection problem \cite{decelle2011sbm,massoullie2014sdp,mossel2013sbm}. Our results apply directly to this model in case the kernel is positive semidefinite. 
\end{exam}

\begin{exam}[Rank 1 model]
The next example concerns the Rank 1 model for random graphs. In this model, we have a function $\psi : [0,1] \to \mathbb{R}^+$ such that 
$K (x,y) = \psi(x) \psi(y)$. $\psi(x)$ governs the ``activity" of the vertex and the probability of a connection is determined by the product of the activities of the two vertices. This model has been extensively studied, see, e.g.,\cite{BDMl2006degree,chunglu2002degree,norrosreittu2006degree}. Prominent features of interest include the existence and size of a giant component, the degree distribution, the typical distances between the vertices etc. We refer to \cite[Section 16]{BJR2007inhomogeneous} for an extensive survey of the related family of models and connections to earlier results. The kernel is positive semidefinite in this case. Further, $\int K < \infty$ whenever
$\psi \in L^1$.  In this case, the kernel is in our framework and our result applies. 

We note that for this example, if $\psi>0$ is constant on blocks, then we reduce to the example of block kernels discussed earlier. The approximation scheme for this example basically approximates the function $\psi$ by a piecewise constant function on $[0,1]$. From an algorithmic viewpoint, our result says that for evaluating the first order correction term, we can coarsen the model to a setup where there are finitely many species with the same activity. 
\end{exam}

\begin{exam}[Dubins's model]
Consider the Dubins kernel $K(x,y) = \frac{1}{\max\{x,y\}}$.  Observe that $K\in L^1$, and is symmetric and positive definite as a function so that our results apply. (The associated integral operator is a bounded operator from $L^2([0,1])$ to $L^2([0,1])$ as shown in \cite{BJR2007inhomogeneous}.) In this case, $\P[ \{i,j \} \in E_N ] = c/j \wedge 1$ for $j \geq i, c$. This corresponds to the situation where the graph is formed by a sequential addition of vertices, and the $j^{th}$ vertex joins to the existing vertices independently with probability $c/j$. This model is naturally inhomogeneous, in that the older vertices usually have higher degrees and play a crucial role in determining the structure of the graph. An infinite version of this model was introduced by Dubins in 1984, who wished to determine the critical $c$ such that the graph would have an infinite path (see \cite{kalikow88weiss, shepp89connectedness}). The critical constant $c=1/4$ was partially determined by Kalikow and Weiss \cite{kalikow88weiss} and finally determined by Shepp \cite{shepp89connectedness}. Durett \cite{durrett03rigorous} determined that $c=1/4$ is also the critical threshold for the emergence of a giant component in the finite graph. We refer the reader to \cite{BJR2007inhomogeneous} for a detailed survey of the model and related results. 
\end{exam}

\subsection*{Acknowledgements}
This research was conducted while A.J. was supported by NSF OISE-1604232 and J.K. was partially supported by
NSERC grant RGPIN-2015-04637. 
\subsection{Outline}
Before turning to the body of the paper, let us briefly outline the structure of the paper and the proof of 
the main results. Theorem \ref{thm:discretization} follows using the interpolation idea introduced in \cite{dembo2016extremal} and generalized in \cite{sen2016optimization}. Its proof, included in Section \ref{sec:cut_proofs}, compares the $\mqcut$ value on inhomogeneous random graphs to the ground state of the inhomogeneous Potts spin glass. 
The proof of Theorem \ref{thm:ParisiUnConstrained} follows the method outlined above.
The starting point of the proof is 
the characterization of a family of arrays that follow a natural generalization of the Ghirlanda-Guerra
identities \cite{GG} for this setting. This combines the synchronization mechanisms of \cite{P1,P2} and is included in 
 \prettyref{sec:characterization}. In \prettyref{sec:perturbation}, we construct a perturbation of the system that
does not affect the limiting free energy but allows us to use the derived invariance properties.
With these results in hand, we prove the upper bound in \prettyref{sec:Guerra-UB} using a Guerra-type interpolation
and the matching lower bound using an Aizenman-Sims-Starr scheme in \prettyref{sec:Aizenman-LB}.
Before proving the lower bound, we briefly study  the continuity of certain functionals
used in the lower bound  in \prettyref{sec:continuity}. 
\section{Invariant arrays and their Characterization}\label{sec:characterization}

In this section, we study an invariance property that combines the
multispecies and vector spin Ghirlanda-Guerra type identities \cite{P1, P2} for the limit points
of doubly infinite array of overlaps, \prettyref{eq:overlap-def}, of independent 
draws from the Gibbs measure. This will allow us to characterize
these limit points. 

Let $\cR_{\kappa}$
be the space of arrays of the form $R_{\ell,\ell'}^s$ such
that there is a collection of vectors $\bigl(v_{i}(\ell,s)\bigr)_{i,\ell\geq1,s\in\sS}$
in $\R^{\kappa}$ such that 
\begin{equation}\label{eq:arraydefinition}
R_{\ell,\ell'}^s=\sum_{i}v_{i}(\ell,s)\tensor v_{i}(\ell',s),
\end{equation}
and such that 
\begin{equation}
R_{\ell,\ell'}^s(e_{k},e_{k})\in[0,1]\quad\forall k\in[\kappa].
\end{equation}
We equip $\cR_{\kappa}$ with the induced topology from the product
topology on a countable product of $[0,1]$ with itself so that, in particular, it is compact Polish. 
Observe that the overlap array \prettyref{eq:overlap-def} is in $\cR_{\kappa}$ for each $N$. 

We now introduce the aforementioned invariance property. For any $m,p\geq1$, $(\nu_{s}^{k})_{s\in\sS,k\in m}\in\R^{\kappa}$, and $\phi:\R^{m}\to\R$,
let
\begin{equation}
\mathcal{Q}_{\ell,\ell'}=\varphi\left( \bigl[ \bigl((R_{\ell,\ell'}^s) ^{\circ p}\nu_{1}^{s},\nu_{1}^{s} \bigr),\dots,\bigl( (R_{\ell,\ell'}^s)^{\circ p}\nu_{m}^{s},\nu_{m}^{s} \bigr) \bigr]_{s \in \sS} \right).
\end{equation} 
The $(R_{\ell,\ell'}^s)^{\circ p}$ term appearing above is the Hadamard $p^{\mathrm{th}}$ power of $R_{\ell,\ell'}^s$. We say that a random variable with values in $\cR_{\kappa}$ is weakly
exchangeable if 
\begin{equation}
R_{\pi_{s}(\ell)\pi_{s}(\ell')}^s\eqdist R_{\ell,\ell'}^s
\end{equation}
for all collections $(\pi_{s})$ of permutations of $\N$ of finitely
many coordinates. We say that a random variable in $\cR$ is $IP$-invariant
if for all $n\geq2$, bounded $f$, and choice of $\mathcal{Q}$ as above, we
have 
\begin{equation}
\E f(R^{n})\mathcal{Q}_{1,n+1}=\frac{1}{n}\E f(R^{n})\E \mathcal{Q}_{1,2}+\frac{1}{n}\sum_{\ell=2}^{n}\E f(R^{n})\mathcal{Q}_{1,\ell},\label{eq:MSPGGI}
\end{equation}
where $R^{n}$ is the array $(R_{\ell,\ell'}^s)_{\ell,\ell'\in[n],s\in\sS}$.

We now turn to the main result of this section. Recall that $\rho^{s}$
is a probability measure on $\sS$, which we denote equivalently by
$\de\rho$. Let 
\begin{equation}
\bar{R}_{\ell,\ell'}=\int R_{\ell,\ell'}^s\,\de\rho.
\end{equation}
Our goal is to prove that random arrays in $\cR$ that satisfy \prettyref{eq:MSPGGI}
synchronize, in the sense that $R_{\ell,\ell'}^s$ is actually a Lipschitz
function of the trace. 
\begin{theo}
\label{thm:IP-characterization}Suppose that $R$ is a $\cR_\kappa$-valued
random variable that is $IP$-invariant and weakly exchangeable. Then
there are deterministic, Lipschitz functions $\Psi_{s}$, depending
on the law of $R$, such that 
\begin{equation}
R_{\ell,\ell'}^s =\Psi_{s} \bigl(\tr (\bar{R}_{\ell,\ell'}) \bigr)
\end{equation}
almost surely. 
\end{theo}
The proof of this result is essentially by composition of the synchronization
theorems from \cite{P1,P2}. 
\begin{lem}
Suppose that $R$ in $\cR$, is weakly exchangeable and $IP$-invariant.
Then there exists deterministic, Lipschitz functions, depending on
the law of $\tr (\bar{R}_{\ell,\ell'})$, such that 
\begin{equation}
\tr(R_{\ell,\ell'}^s)=L_{s}\bigl(\tr (\bar{R}_{\ell,\ell'}) \bigr)
\end{equation}
\end{lem}
\begin{proof}
We begin by observing that for any element of $\cR$, the array of
traces $\bigr(\tr (R_{\ell,\ell'}^s ) \bigr)_{\ell, \ell' \geq 1}$ is a Gram matrix for each $s$. To
see this simply observe that if we let

\[
V_{\ell}(s)=\sum_{i}v_{i}(\ell,s)\tensor e_{i}
\]
then 
\[
\tr (R_{\ell,\ell'}^s)=\left(V_{\ell}(s),V_{\ell'}(s)\right)_{HS}.
\]
Thus $T=\bigl(\tr (R_{\ell,\ell'}^s) \bigr)$ is a Gram-De Finetti array for each
$s\in\sS.$ Applying \eqref{eq:MSPGGI} with $f$ and $\varphi$ being
functions of $T$, we see that $T$ satisfies the Multispecies Ghirlanda-Guerra
Identities \cite[Eq. (36)]{P1}
\[
\E f(T^{n}) \mathcal{Q}_{1,n+1}=\frac{1}{n}\E f(T^n)\cdot\E \mathcal{Q}_{1,2}+ \frac{1}{n} \sum_{\ell=2}^{n}\E f(T^{n})\mathcal{Q}_{1,\ell}.
\]
Recall from \cite[Section 4]{P1} that for such arrays there exist
Lipschitz functions $L_{s}$, that depend on the law of $T^{n}$,
such that
\[
T_{\ell,\ell'}^s=L_{s}\bigg(\int T_{\ell,\ell}^s\,\de\rho\bigg)
\]
almost surely. Applying this result to our setting yields a family
of Lipschitz functions depending on the law of $\tr (\bar{R}_{\ell,\ell'}) $
such that 
\[
\tr (R_{\ell,\ell'}^s) =L_{s}\bigl(\tr (\bar{R}_{\ell,\ell'}) \bigr)
\]
almost surely, as desired.
\end{proof}
\begin{lem}
Suppose that $R$ in $\cR_\kappa$, is weakly exchangeable and $IP$-invariant.
Then there exist deterministic, Lipschitz, monotone functions $\Theta^{s}:\R_{+}\to\Gamma_{\kappa}$
which depend on the law of $R$ such that 
\[
R_{\ell,\ell'}^s=\Theta^{s}\bigl(\tr (R_{\ell,\ell'}^s) \bigr)
\]
almost surely.
\end{lem}
\begin{proof}
For fixed $s\in\sS$, we may apply \eqref{eq:MSPGGI}, with $\mathcal{Q}$ of
the form 
\[
\mathcal{Q}_{\ell,\ell'}=\varphi\Big(  \bigl((R_{\ell,\ell'}^s) ^{\circ p}\nu_{1}^{s},\nu_{1}^{s} \bigr),\dots,\bigl( (R_{\ell,\ell'}^s)^{\circ p}\nu_{m}^{s},\nu_{m}^{s} \bigr)  \Big).
\]
As a result, taking $f$ to be a function of this species as well,
we see that the array satisfies 
\[
\E f\bigl(R^{n}(s)\bigr)\mathcal{Q}_{1,n+1}=\frac{1}{n}\E f\bigl(R^{n}(s) \bigr)\E \mathcal{Q}_{1,2}+\frac{1}{n}\sum_{\ell=2}^{n}\E f\bigl(R^{n}(s)\bigr)\mathcal{Q}_{1,\ell},
\]
where $R^n(s) = (R^s_{\ell, \ell'})_{\ell, \ell' \in [n]}$. It was shown in \cite[Theorem 3]{P2} that for such arrays,
there is a Lipschitz, monotone function $\Theta^{s}$ depending on
the law of $R^{n}(s)$ such that 
\[
R_{\ell,\ell'}^s=\Theta^{s} \bigl( \tr (R_{\ell,\ell'}^s) \bigr)
\]
almost surely, as desired.
\end{proof}
\begin{proof}[\textbf{\emph{Proof of \prettyref{thm:IP-characterization}}}]
 Applying the previous two lemmas we obtain families $(\Theta^{s})_{s \in \sS}$ and $\left(L_{s}\right)_{s \in \sS}$. The result then follows by taking 
\[
\Psi_{s}=\Theta^{s}\circ L_{s}.
\]
\end{proof}

\section{Perturbation for Invariance}\label{sec:perturbation}
In this section, we show that after a small perturbation, the limiting overlap array will satisfy a generalized form of the Ghirlanda Guerra identities appearing in \cite{P1} and \cite{P2}. This argument is standard and can be safely skipped by the expert reader.
For completeness we include it here.  The key observation is that, as with the Potts model, it is crucial that we restrict ourselves to configurations with fixed proportions of states. 

Let $h_\theta(\sigma)$ be a Gaussian process with covariance
\begin{equation}\label{eq:CovPerturbed}
 C^\theta_{\ell, \ell'} = \Cov \bigl( h_\theta(\sigma^{\ell})h_\theta(\sigma^{\ell'}) \bigr)= \prod_{s \in \sS} \prod_{j \leq m} \Big(   \bigl(( R_{\ell, \ell'}^s)^{\circ p } \nu_j^s, \nu_j^s \bigr) \Big)^{n_j^s}.
\end{equation}
where $\theta = (m , p_s, n_1^s, \dots, n_m^s, \nu_1^s, \dots, \nu_m^s)_{s \in \sS}$ are the parameters in the covariance. The Gaussian process $h_\theta(\sigma)$ can be constructed explicitly using a similar construction as in \cite[Section 5]{P2}. We will provide a brief non-constructive existence proof here.
\begin{lem}
	The covariance structure $C^\theta_{\ell, \ell'}$ is positive semidefinite for all $(R^s_{\ell,\ell'})_{s \in \sS} \in \cR_{\kappa}$.
\end{lem}
\begin{proof}
	Clearly if $R^s_{\ell,\ell'} \in \cR_{\kappa}$, then $( R_{\ell, \ell'}^s)^{\circ p } \in \cR_{\kappa}$ for all $p \geq 1$. By the definition on \eqref{eq:arraydefinition}, we can find some collection of vectors $\bigl(v_{i}(\ell,s)\bigr)_{i,\ell\geq1,s\in\sS}$ such that
	\[
	( R_{\ell, \ell'}^s)^{\circ p } = \sum_{i}v_{i}(\ell,s)\tensor v_{i}(\ell',s) = \sum_{i}v_{i}(\ell,s)  v^{\mathrm{T}}_{i}(\ell',s).
	\]
	Given $\nu^s_j \in \R^\kappa$, we have 	
	\begin{equation}\label{eq:gram}
	\bigl(( R_{\ell, \ell'}^s)^{\circ p } \nu_j^s, \nu_j^s \bigr)_{\ell, \ell' \geq 1} = \Bigl( \sum_{i}v_{i}(\ell,s)  v^{\mathrm{T}}_{i}(\ell',s) \nu_j^s, \nu_j^s \Bigr)_{\ell, \ell' \geq 1} = \sum_{i} \bigl( v_{i}(\ell',s)^{\mathrm{T}} \nu_j^s, v_{i}(\ell,s)^{\mathrm{T}} \nu_j^s \bigr)_{\ell, \ell' \geq 1}
	\end{equation}
	is a Gram array and hence positive semidefinite. Since Hadamard products preserves positivity,
	\[
	\prod_{s \in \sS} \prod_{j \leq m} \Big(   \bigl(( R_{\ell, \ell'}^s)^{\circ p } \nu_j^s, \nu_j^s \bigr) \Big)^{n_j^s}
	\]
	is positive semidefinite because it is the Hadamard product of finitely many arrays of the form \eqref{eq:gram}. Hence there exists a Gaussian process indexed with $\sigma^\ell$ with covariance given by \eqref{eq:CovPerturbed}.
\end{proof}

Let $\nu^s_j$ take rational values in $[-1,1]^\kappa$ and define the space of parameters
\begin{equation}
\Theta := \left\{ \theta : m,p, n_1^s, \dots, n_m^s \in \N, \nu_1^s, \dots, \nu_m^s \in \mathbb{Q} \cap [-1,1] \text{ for all $s \in \sS$} \right\}.
\end{equation} 
Since $\Theta$ is countable,  we can find a enumeration map $j(\theta) : \Theta \to \N$. Let $(u_\theta)_{\theta \in \Theta}$ be i.i.d uniform random variables on $[1,2]$ and let $(h_\theta)_{\theta \in \Theta}(\sigma)$ of be pairwise independent copies of $h_\theta$. Finally, define
\begin{equation}
h_N(\sigma) = \sum_{\theta \in \Theta} \frac{1}{2^{j(\theta)}} u_{\theta} h_{\theta} (\sigma).
\end{equation}

Let $d_N \in \sD_N$ be such that $d_N \to d \in \sD$, and consider the perturbed Gibbs measure on $\Sigma_N(d_N)$ given by
\begin{equation}\label{eq:Perturbed-Gibbs-Measure}
G^{pert}_{d_N} = \frac{\exp H_N^{pert} (\sigma)}{Z_N(d_N)}, ~H_N^{pert} = H_N(\sigma) + s_N h_N(\sigma), 
\end{equation}
where $\sigma \in \Sigma_N(d_N)$ and $s_N = N^\alpha$ for $1/4 < \alpha < 1/2$. We then have the following.
\begin{theo}\label{thm:MSPOTTSGG}
There is a choice of $(u_\theta)$ such that the following holds
\begin{itemize}
\item The perturbation is small in the sense that
\[
\lim_{N \to \infty}\abs{\frac{1}{N} \E \log \sum_{\sigma \in \Sigma_N(d)} \exp(H_N(\sigma))-\frac{1}{N} \E \log \sum_{\sigma \in \Sigma_N(d)} \exp(H^{pert}_N(\sigma))}=0
.\]
\item  If  $ R^N = (R^s_{\ell,\ell'})_{s \in \sS, \ell,\ell' \geq 1}$ is the overlap array drawn from $\E (G^{pert}_{d_N})^{\infty}$, then any weak limit point, $ R^\infty$, satisfies \eqref{eq:MSPGGI}. 
\end{itemize}
\end{theo}
\begin{proof} 
The proof of this fact is almost identical to Chapter 3.2 in \cite{P3}, so we omit most details. The only essential difference
is the same as that in \cite{P2}, namely to point out why restricting the configuration space is important. 
This is because the main integration by parts step in the proof of \cite[Theorem 3.2]{P3} uses in an essential
way that the self overlap, $R_{\ell, \ell}$, and thus the variance of the field $h_N$, is constant. In our setting,
the relevant term, namely $C_{\ell,\ell}^\theta$ is plainly constant on $\Sigma_N(d_N)$ by inspection of \prettyref{eq:CovPerturbed}.

\end{proof}
\section{Upper Bound - Guerra Interpolation}\label{sec:Guerra-UB}
We now turn to proving the upper bound for the restricted free energy in \eqref{eq:ParisiConstrained} by a Guerra
interpolation argument \cite{Guerra}. Recall the definition of $\Sigma_N^\eps(d)$ \eqref{eq:constrainedspace} and let the corresponding partition function 
be denoted by $Z_N^\epsilon(d) = \sum_{\sigma \in \Sigma_N^\epsilon(d)} \exp\bigl(H_N(\sigma)\bigr)$. We will prove
\begin{equation}
\limsup_{N \to \infty} \frac{1}{N} \E \log Z_N^\epsilon(d) \leq \sP\bigl(r, x, d, (\lambda^s,Q^s)_{s \in \sS}\bigr)+ O(\epsilon).
\end{equation}
Given a sequence of strictly increasing $(x_i)_{i=1}^r$ as in \eqref{eq:RPC6}, let $(v_\alpha)_{\alpha \in \N^r}$ be the weights of the Ruelle probability cascades associated with that sequence. For $\alpha, \beta \in \N^r$, define
\begin{equation}
|\alpha \wedge \beta| = \min\{ 0 \leq p \leq r-1 \mathrel{}\mid\mathrel{} \alpha_1 = \beta_1, \dots,  \alpha_p = \beta_p, \alpha_{p+1} \neq \beta_{p+1} \}
\end{equation}
and $|\alpha \wedge \beta| =r$ if $\alpha = \beta$. For each species, let $(Z_{s}^{\alpha}(k))_{k\leq\kappa}$ be the centered Gaussian vector with covariance 
\begin{equation}\label{eq:z-def}
\E Z_{s}^{\alpha}Z_{s}^{\beta} = 2\sum_{t \in \sS}\Delta_{st}^{2}\rho_N^{t} Q_{\abs{\alpha\wedge\beta}}^{t}.
\end{equation}
Similarly let 
\begin{equation}\label{eq:y-def}
\E Y^{\alpha}Y^{\beta}=\sum_{s,t \in \sS}\Delta_{st}^{2}\rho_N^{s}\rho_N^{t}\left(Q_{\abs{\alpha\wedge\beta}}^{s},Q_{\abs{\alpha\wedge\beta}}^{t}\right).
\end{equation}
For each $s\in\sS$ and each $i\in I_s$, let $Z_i^\alpha$ be an independent copy of $Z_s^\alpha$. The processes $Z_i^\alpha$, $Y^\alpha$, and $H_N(\sigma)$ are all independent. Finally 
define the interpolating Hamiltonian, 
\begin{equation}
H_{N}(\sigma,\alpha;t)=\sqrt{t}H_{N}(\sigma)+\sqrt{t}\sqrt{N}Y^{\alpha}+\sqrt{1-t}\sum_{i \leq N}Z_{i}^{\alpha}(\sigma_{i}),
\end{equation}
and the corresponding interpolating free energy function
\begin{equation}
\phi^\epsilon_{N}(t)=\frac{1}{N}\E\log\sum_{\alpha \in \N^r }v_{\alpha}\sum_{\sigma\in \Sigma^\epsilon_N(d)}e^{H_{N}(\sigma,\alpha;t)}.
\end{equation}
We then have the following result.
\begin{lem}\label{lem:GuerraInterpolation}
	For any $\epsilon > 0$ and $N \geq \frac{1}{\epsilon}$,
	\[
	\partial_t \phi^\epsilon_N(t) \leq   C \epsilon^2.
	\]
	for some constant $C(\kappa, \Delta, |\sS|)$, uniformly in $N$. 
	Furthermore, if $d\in\sD_N$ and $\epsilon=0$, then
	\[
	\partial_t \phi_N^0(t)\leq 0.
	\]
\end{lem}
\begin{proof}
	For any $d \in \sD$, the set $\Sigma_N^\epsilon(d)$ is non-empty for $N \geq \frac{1}{\epsilon}$. Recall that by Gaussian integration by parts, 
	\begin{align}
	\partial_t \phi^\epsilon_N(t)  & = \frac{1}{N}\E\Big\langle \partial_{t}H(\sigma,\alpha)\Big\rangle \notag\\
	& = \frac{1}{N} \E\Big\langle \E \partial_{t}H(\sigma^{1},\alpha^{1})H(\sigma^{1},\alpha^{1})- \E\partial_{t}H(\sigma^{1},\alpha^{1})H(\sigma^{2},\alpha^{2}) \Big \rangle \label{eq:upbdtemp}
	\end{align}
	where $\langle \cdot \rangle$ is with respect to the Gibbs measure $G(\sigma, \alpha)\propto v_\alpha \exp(H_N(\sigma, \alpha;t))$ on $\Sigma_N^\epsilon(d) \times \N^r$. If we write $Z^\alpha_i(\sigma_i) =  \sum_{k \leq \kappa}Z_i^\alpha(k) \1 (\sigma_i = k)$ then 
	\begin{align*}
	\E\partial_{t}H(\sigma^{1},\alpha^{1})H(\sigma^{2},\alpha^{2})	& =\frac{1}{2}\Big(\E H_{N}(\sigma^{1})H_{N}(\sigma^{2})+N\E Y^{\alpha^{1}}Y^{\alpha^{2}}-\sum_{s \in \sS}\sum_{i\in I_{s}}\E Z_{s}^{\alpha^1}(\sigma_{i}^{1})Z_{s}^{\alpha^2}(\sigma_{i}^{2})\Big)\\
	& =\frac{N}{2}\sum_{s,t \in \sS} \rho_N^s \rho_N^t\Delta_{st}^{2}\Big[\big( R_{1,2}^s, R_{1,2}^t\big)+\big(Q_{\abs{\alpha^1\wedge\alpha^2}}^{t},Q_{\abs{\alpha^1\wedge\alpha^2}}^{s}\big)- 2\big(Q_{\abs{\alpha^1\wedge\alpha^2}}^{t}, R^s_{1,2}\big)\Big]\\
	& =\frac{N}{2}\sum_{s,t\in \sS}\Delta_{st}^{2}\left( R^s_{1,2}-Q_{\abs{\alpha^1\wedge\alpha^2}}^{s}, R^t_{1,2}-Q_{\abs{\alpha^1\wedge\alpha^2}}^{t}\right)\rho_N^{s}\rho_N^{t}\geq0.
	\end{align*}
	The last quantity is non-negative comes from the fact that if we
	define the matrix 
	\[
	A_{1,2,\alpha^1,\alpha^2}=\left(\left( R^s_{1,2} -Q_{\abs{\alpha^1\wedge\alpha^2}}^{s}, R^t_{1,2}-Q_{\abs{\alpha^1\wedge\alpha^2}}^{t}\right)\right)_{s,t}
	\]
	then it is positive definite as it is a Gram-Matrix. The Hadamard product of this matrix with the positive definite matrix $\Delta = (\Delta^2_{s,t})_{s,t}$, is still positive definite by the Schur product theorem. Thus the above expression can be written as
	\[
	\frac{N}{2}\left(\Delta^{2}\circ (A_{1,2,\alpha^1,\alpha^2})\rho,\rho\right)\geq0,
	\]
	where $\circ$ denotes the Hadamard product.
	For $\sigma \in \Sigma_N^\epsilon(d)$, the self overlap matrix $R^s_{1,1}$ is a diagonal matrix with entries $R^s_{1,1}(k,k) \in [d^s_k - \epsilon, d_k^s + \epsilon]$. Therefore the diagonal terms of $R^s_{1,1} - Q^s_{|\alpha^1, \alpha^1|}$ satisfy
	\[
	|R^s_{1,1}(k,k) - Q^s_{r}(k, k)| \leq \epsilon, \quad \forall s \in \sS, k \leq \kappa.
	\]
	Since there are $\kappa$ non-zero terms in the inner product $\left( R^s_{1,1} -Q_{\abs{\alpha^1\wedge\alpha^1}}^{s}, R^t_{1,1}-Q_{\abs{\alpha^1\wedge\alpha^1}}^{t}\right)$, our upper bound of the diagonals imply
	\[
	\E\partial_{t}H(\sigma^{1},\alpha^{1})H(\sigma^{1},\alpha^{1}) = \frac{N}{2}\left(\Delta^{2}\circ (A_{1,1,\alpha^{1}\alpha^{1}})\rho,\rho\right) \leq \frac{N\|\Delta\| |\sS|^2\kappa}{2} \epsilon^2.
	\]
	Putting this back in \eqref{eq:upbdtemp}, implies 
	\[
	\partial \phi^\epsilon_N (t) = \E\left\langle\frac{1}{2}\left(\Delta^{2}\circ (A_{1,1,\alpha^{1},\alpha^{1}})\rho,\rho\right)\right\rangle -\E\left\langle\frac{1}{2}\left(\Delta^{2}\circ (A_{1,2,\alpha^{1},\alpha^{2}})\rho,\rho\right)\right\rangle\leq  \frac{\|\Delta\| |\sS|^2\kappa}{2} \epsilon^2.
	\]
	In the case $d \in \sD_N(d)$, the set $\Sigma_N(d) = \Sigma_N^0(d)$ is non-empty. For $\sigma \in \Sigma_N(d)$, the overlap array $R^s_{\ell, \ell'} = Q^s_{r}$ for all $s \in \sS$. In particular,  $A_{1,1,\alpha^{1},\alpha^{1}} = 0$, which implies
	\[
	\partial \phi^0_N (t) = -\E\left\langle\frac{1}{2}\left(\Delta^{2}\circ (A_{1,2,\alpha^{1},\alpha^{2}})\rho,\rho\right)\right\rangle\leq  0.
	\]
	
\end{proof}
As a consequence $\frac{1}{N}\E\log Z^\epsilon_{N}(d)$ is bounded above by
\begin{equation}\label{eq:upbdtemp2}
\frac{1}{N}\E\log \sum_{\alpha \in \N^r} v_\alpha \sum_{\sigma \in \Sigma^\eps_N(d)} \exp\sum_{i \leq N} Z^\alpha_{i} (\sigma_i)  - \frac{1}{N}\E\log \sum_{\alpha \in \N^r} v_\alpha \exp \sqrt{N} Y^\alpha + O(\epsilon^2).
\end{equation}
We now introduce the Lagrange multipliers $(\lambda_k^s)_{k \leq \kappa} \in \R^k$, which are
dual to the proportions $\sum_{i\in I_s} \1(\sigma_i=k)$. 
If we add and subtract $\sum_{i \leq N}\sum_{k \leq \kappa}\lambda^{s(i)}_{k} \1(\sigma_i = k)$ in the exponent of the first term in \eqref{eq:upbdtemp2}, then the first term is bounded by
\[
-\sum_{s \in \sS}\sum_{k \leq \kappa} \rho_N^s d^s_k \lambda^s_k + \frac{1}{N}\E\log \sum_{\alpha \in \N^r} v_\alpha \sum_{\sigma \in \Sigma^\eps_N(d)} \exp\sum_{i \leq N} \Big(Z^\alpha_{i} (\sigma_i) + \sum_{k \leq \kappa} \1(\sigma_i = k) \lambda_k^{s(i)} \Big)   + O(\epsilon).
\]
Since $\Sigma_N^\epsilon(d) \subset \Sigma_N$, summing over $\sigma \in \Sigma_N$ in the larger set only increases our upper bound. We can now factor into species using the basic properties of the Ruelle Probability Cascades (see the discussion after Theorem 2.9 in \cite{P3}) to see
\begin{align}
&\quad \frac{1}{N}\E\log \sum_{\alpha \in \N^r} v_\alpha \sum_{\sigma \in \Sigma_N} \exp\sum_{i \leq N} \Big( Z^\alpha_{i} (\sigma_i) + \sum_{k \leq \kappa} \1(\sigma_i = k) \lambda_k^{s(i)}   \Big) \notag\\
&= \frac{1}{N}\E\log \sum_{\alpha \in \N^r} v_\alpha \prod_{i \leq N} \sum_{\sigma_i \leq \kappa} \exp  \Big(Z^\alpha_{i} (\sigma_i) + \sum_{k \leq \kappa} \1(\sigma_i = k) \lambda_k^{s(i)} \Big)  \notag\\
&= \sum_{s \in \sS} \rho_N^s \E\log \sum_{\alpha \in \N^r} v_\alpha \sum_{\sigma \leq \kappa} \exp  \Big(Z^\alpha_s (\sigma) +  \lambda_\sigma^s \Big)\notag\\
&= \sum_{s \in \sS} \rho_N^s X_0^s.\label{eq:upbdzterm}
\end{align}
where $X_0^s$ was defined in \eqref{eq:recursion}. Similarly, we see
\begin{equation}\label{eq:upbdyterm}
\frac{1}{N}\E\log \sum_{\alpha \in \N^r} v_\alpha \exp \sqrt{N} Y^\alpha = \frac{1}{2} \sum_{\ell = 0}^{r-1} x_{\ell} \sum_{s,t \in \sS} \Delta^2_{st} \rho_N^s \rho_N^t \left( (Q^s_{\ell+1},Q^t_{\ell+1}) - (Q^s_{\ell},Q^t_{\ell}) \right).
\end{equation}
Referring back to \eqref{eq:upbdtemp2}, equations \eqref{eq:upbdzterm} and \eqref{eq:upbdyterm} imply
\begin{equation}\label{eq:guerra-ub-thick}
\limsup_{N \to \infty} \frac{1}{N} \E \log Z_N^\epsilon(d) \leq \sP\bigl(r, x, d, (\lambda^s,Q^s)_{s \in \sS}\bigr)+ O(\epsilon).
\end{equation}

In the case $d_N \in \sD_N$, $\sum_{i \in I_s} \1(\sigma_i = k) = N_s d_k^s$ for all $s \in \sS$ and $k \leq \kappa$. The above computation implies
\begin{equation}\label{eq:guerra-ub}
\frac{1}{N} \E \log Z_N(d_N) \leq \sP\bigl(r, x, d, (\lambda^s,Q^s)_{s \in \sS}\bigr).
\end{equation}

\section{Continuity and Decoupling Theorems}\label{sec:continuity}
Before we turn to the proof of the matching lower bounds, we briefly pause to study
the analytical properties of some of the functionals used in the upper bound, 
as well as relevant functionals for the lower bound. 
These functionals will  be in terms of the Gaussian processes $(Z_s^\alpha)$ and $(Y^\alpha)$
from \eqref{eq:z-def} and \eqref{eq:y-def} respectively.

Many of the proofs in this section are essentially identical to arguments either from 
 \cite[Section 3]{P2}, or are standard arguments and can be seen, for example, in \cite{P3}. Thus to make the presentation 
 concise, we explain only the parts where these arguments deviate from standard arguments
 and outline the rest. 

\subsection{Decoupling Size of the Constraints}\label{thm:decouple} 
When we compute the lower bound for the free energy, we will find that the cavity
method will naturally impose an additional constraint on free energy, namely,
that the cavity coordinates satisfy the additional constraint that they lie in some $\Sigma(d)$.
As this constraint does not appear in the upper bound from \prettyref{sec:Guerra-UB}, we will need to remove this to obtain 
the matching lower bound. To this end, define the functional on $\sD_M$
\begin{equation}\label{eq:func-constrained}
f^s_{M}(d) := \frac{1}{M} \E \log \sum_{\alpha \in \N^r} v_\alpha \sum_{\sigma \in \Sigma_M(d)} \exp \sum_{i \leq M} Z_i^\alpha(\sigma_i),
\end{equation}
where $Z_i^\alpha$ are i.i.d. copies of $Z_s^\alpha$. Furthermore, we fix the covariance structure of $Z_i^\alpha$ and make the dependence of $X_0^s$ on the parameter $\lambda$ explicit
\begin{equation}\label{eq:decouplefunctional}
X_0^s\bigl(\lambda\bigr) = \frac{1}{M} \E \log \sum_{\alpha \in \N^r} v_{\alpha} \sum_{\sigma \in \Sigma_M} \exp \sum_{i \leq M} Z_{s}^\alpha(\sigma_i) + \sum_{k \leq \kappa} \lambda^s_k\1(\sigma_i = k).
\end{equation}
We have the following result.
\begin{theo}
If $d_M^s \in \sD_M$ and $\lim_{M \to \infty} d_M^s \rightarrow d^s$ then
\begin{equation}
\lim_{M \to \infty} f^s_M(d_M) = \inf_{\lambda^s} \bigg(- \sum_{k \leq \kappa} \lambda_k^s d_k^s + X_0^s\bigl(\lambda\bigr) \bigg).
\end{equation}
\end{theo}

\begin{proof}
We begin by observing that there is a constant $L$ such that
\begin{equation}\label{eq:epsilon-dilation-distance}
\sup_{d \in \sD_M} |f^s_{M, \epsilon} (d) - f^s_M(d) | \leq L \sqrt{\epsilon},
\end{equation}
where $f^s_{M, \epsilon} (d)$ is the same functional as in \eqref{eq:func-constrained} but summed over $\sigma \in \Sigma^{\epsilon}_M(d)$. This is the analogue of \cite[Lemma 3]{P2}, adapted to the covariance structure of the Gaussian processes $Z_i$ defined in \eqref{eq:z-def}. We also fix the covariance structure of $Z_i$ to remove its dependence on $d$. Let $\sigma\in \Sigma_M^\eps(d)$ and $\tilde\sigma$ be a vector in $\Sigma_M(d)$ 
with the minimal number of different coordinates from $\sigma$. 
Let 
\[
\tilde f^s_{M,\epsilon} = \frac{1}{M} \E \log \sum_{\alpha \in \N^r} v_\alpha \sum_{\sigma \in \Sigma^\epsilon_M (d)} \exp \sum_{i \in S_M} Z_i^\alpha(\tilde \sigma_i).
\]
 Finally, let $\tilde Z_i^\alpha$ be independent copies of $Z_i^\alpha$, and consider the ``smart path''
 \[
 Z_t(\alpha,\sigma)= \sum_{i\leq M} (\sqrt{t} Z_i^\alpha(\sigma)+\sqrt{1-t}\tilde Z_i^{\alpha}(\tilde\sigma_i).
 \]
We will then show that
 \[
 \phi(t) = \frac{1}{M} \E \log \sum_{\alpha \in \N^r} v_\alpha \sum_{\sigma \in \Sigma^\epsilon_M(d)} \exp Z_t(\alpha, \sigma),
\]
has a derivative that is bounded by $C\epsilon$. To see this, let 
\begin{align*}
	C\left((\sigma^1, \alpha^1), (\sigma^2, \alpha^2)\right) &= \frac{1}{M} \frac{\partial Z_{t} (\sigma^1, \alpha^1)}{\partial_t} Z_{t} (\sigma^2, \alpha^2) 
	\\&= \frac{1}{M} \sum_{i \leq M} \left(\sum_{t \in \sS}\Delta_{st}^{2}\rho^{t} \Big( Q_{\abs{\alpha\wedge\beta}}^{t} (\sigma_i^1, \sigma_i^2) - Q_{\abs{\alpha\wedge\beta}}^{t} (\tilde \sigma_i^1, \tilde \sigma_i^2) \Big)\right) 
	\\&\leq  \|\Delta\| \cdot |\sS| \cdot \kappa\cdot \epsilon.
\end{align*}
In the last line, we use that $Q^t\leq D^t$ and that $\norm{D^t}\leq 1$ along with the observation that
\[
\sum \1\left(\sigma_i\neq\tilde{\sigma}_i\right) \leq \lceil \kappa M \epsilon\rceil.
\]
Differentiating and integrating by parts, we obtain
\[
\abs{\phi'(t)} =\abs{\E\left\langle C((\sigma^1,\alpha^1),(\sigma^1,\alpha^1))-C((\sigma^1,\alpha^1),(\sigma^2,\alpha^2))\right\rangle}
\leq 2L\epsilon.
\]
for some constant $L=L(\kappa,\sS,\Delta)$. Integrating this inequality, we can conclude
\begin{equation}\label{eq:thm11temp}
|f^s_{M, \epsilon} - \tilde f^s_{M, \epsilon}| \leq 2 L \epsilon.
\end{equation}

For $\sigma\in \Sigma(d)$, let us denote by $\mathcal{N}(\sigma)$ the number of configurations $\rho\in\Sigma_\eps(d)$ such that $\tilde{\rho} = \sigma.$ Then we can rewrite and bound $\tilde{f}^s_{M,\eps}(d)$ as follows,
\begin{align*}
\tilde{f}_{M,\eps}(d) 
&= 
\frac{1}{M} \E \log \sum_{\alpha\in\N^r} v_\alpha \sum_{\sigma\in \Sigma(d)} \mathcal{N}(\sigma)\exp \beta \sum_{i\leq N} Z_{i}^\alpha({\sigma}_i).
\\
&\leq f^s_{M}(d) +\frac{1}{N}\max_{\sigma\in \Sigma(d)} \log \mathcal{N}(\sigma).
\end{align*}
Using a combinatorial argument, the term containing $\mathcal{N}(\sigma)$ can be made arbitrarily small by choosing $\eps$ small enough. For any $\sigma\in \Sigma(d)$, the number $\mathcal{N}(\sigma)$ is bounded by the number of configurations $\rho$ such that $\sum_{i\leq M}I(\rho_i\not= {\sigma}_i) \leq L M\eps$. By the classical large deviation estimate for Bernoulli random variables, a number of different ways to choose $LM\eps$ coordinates is bounded by  $2^M \exp(-M I(1-L\eps)),$ where 
$$
I(x)=\frac{1}{2}\bigl((1+x)\log(1+x)+(1-x)\log(1-x)\bigr),
$$
and there are $\kappa^{LM\eps}$ ways to choose $\rho_i$ different from $\sigma_i$ on these coordinates. Therefore,
\begin{align*}
\frac{1}{M}\max_{\sigma\in \Sigma(d)} \log \mathcal{N}(\sigma)
&\leq
L\eps\log\kappa+
\log 2 - I(1-L\eps)
\\
&=
L\eps\log\kappa+
\log\Bigl(1+\frac{L\eps}{2-L\eps}\Bigr)+\frac{L\eps}{2}\log
\frac{2-L\eps}{\eps}
\leq L\sqrt{\eps},
\end{align*}
for small enough $\eps$. We showed that ${f}_{M,\eps}(d)\leq f_M(d)+L\sqrt{\eps}$.
Combining this with \eqref{eq:thm11temp} yields the estimate
\[
|f^s_M - f^s_{M,\epsilon}|\leq 2 L \sqrt{\epsilon}.
\]
Note that \eqref{eq:epsilon-dilation-distance}, implies that the map $d\mapsto f_M^s(d)$ is
H\"older 1/2. 

The remaining steps to complete the proof are identical to \cite[Section 3]{P2}. An additivity argument and \eqref{eq:epsilon-dilation-distance} will imply the following.
\begin{lem}\label{lem:decouple1}
\cite[Lemma 4,5]{P2} If $d_M \in \sD_M$ and $\lim_{M \to \infty} d_M = d \in \sD$ then the limit
\[
f(d) := \lim_{M \to \infty} f_M(d_M) = \lim_{\epsilon \to 0} f_{\epsilon}(d)
\]
exists and is concave. In addition, for all $d^1, d^2 \in \sD$
\[
| f(d^1) - f(d^2)| \leq L \|d^1 - d^2\|_\infty^{1/2}
\]
for some constant $L$ that depends $\kappa, \sS, \Delta$. 
\end{lem}
Next, a direct computation using the recursive property of the Ruelle Probability cascades will imply the following result.
\begin{lem}\label{lem:decouple2} \cite[Lemma 6]{P2} For any $\lambda = (\lambda_k)_{k \leq \kappa} \in \R^\kappa$
	\begin{equation}\label{eq:biconjugation}
	X_0^s(\lambda) = \max_{d \in \sD} \Bigl( f^s(d) + \sum_{k \leq \kappa} \lambda_k^s d_k^s \Bigr).
	\end{equation}
\end{lem}

Notice $f^s_M(d)$ is continuous and bounded on $\sD$. Since it is also concave by \prettyref{lem:decouple1}, we can take the Legendre transform of \eqref{eq:biconjugation} to obtain
\[
\lim_{N \to \infty} f^s_M (d) = \inf_\lambda \bigg( - \sum_{k \leq \kappa} \lambda^s_kd_k^s + X_0^s(\lambda) \bigg).
\]

\end{proof}

\subsection{Continuity Theorems}
Along with the decoupling theorem, we will also need to prove continuity
of the functionals appearing in the Aizenman-Sims-Starr scheme for this system. 

Define the two functionals on a subset $S \subset [\kappa]^M$,
\begin{equation}\label{eq:func-pi-z}
f^Z_{M}(S,s; \bm \pi) := \frac{1}{M_s} \E \log \sum_{\alpha \in \N^r} v_\alpha \sum_{\sigma \in S} \exp \sum_{i \leq M_s} Z_i^\alpha(\sigma_i),
\end{equation}
\begin{equation}\label{eq:func-pi-y}
f^Y_{M}(\bm \pi) := \frac{1}{M} \E \log \sum_{\alpha \in \N^r} v_\alpha \exp \sqrt{M} Y^\alpha,
\end{equation}
where $(Z_i^\alpha)_{i \leq M}$ are \iid copies of $Z_s^\alpha$ defined in \eqref{eq:z-def}. Observe that  this functional depends on 
$\bm \pi = (\pi^s)_{s \in \sS}$ is through the covariance structures \eqref{eq:z-def} and \eqref{eq:y-def}.
\begin{lem}\label{lem:ctyofz}
For every $S \subset [\kappa]^M$, the functional $f^Z_{M}(S,s; \bm \pi)$ is Lipschitz in $\bm \pi$
\[
|f^Z_{M}(S,s; \bm \pi) - f^Z_{M}(S,s; \bm {\tilde \pi}) | \leq L \int_{0}^{1} \max_{s \in \sS} \| \pi^s(x) - \tilde\pi^s(x) \|_1 \, \de x.
\]
\end{lem}

\begin{proof} 
Observe that  any two monotone paths $\bm \pi$ and $\tilde{\bm \pi}$,
can be associated with a single sequence
\[
x_{-1} = 0 \leq x_0 \leq \dots \leq x_{r-1} \leq x_r = 1\\
\]
and, for every $s\in\sS$, two sequences
\begin{align*}
	0 < Q^s_0 \leq \dots < Q^s_{r-1} < Q^s_r = \operatorname{diag}(d_1, \dots, d_\kappa)\\
	0 < \tilde Q^s_0 \leq \dots < \tilde Q^s_{r-1} < \tilde Q^s_r = \operatorname{diag}(d_1, \dots, d_\kappa).
\end{align*}
Consider the Gaussian processes $Z_s^\alpha$ and $\tilde Z_s^\alpha$ with the covariance
\prettyref{eq:z-def} with $Q$ and $\tilde Q$ respectively, and consider the smart path between these
two processes 
\[
Z_{t,i}^\alpha =\sqrt{t}Z_i^\alpha + \sqrt{1-t}\tilde Z_i^\alpha.
\]
If we take $M$ copies of this process, and let $(v_\alpha)$ be the Ruelle Probability Cascade \cite{Ruelle}
associated to $(x_k)$, then if we define
\[
\varphi(t) := \frac{1}{M_s} \E \log \sum_{\alpha \in \N^r} v_\alpha \sum_{\sigma \in S} \exp \sum_{i \leq M_s} Z_{t,i}^\alpha (\sigma_i),
\]
we have $\phi(1)=f_Z(S,s;\bm\pi)$ and $\phi(0)=f_{Z}(S,s;\tilde{\bm{\pi}})$. If $H_{M,t}(\alpha, \sigma) = \sum_{i \leq M} Z_{t,i}^\alpha(\sigma_i)$, let $\langle \cdot \rangle_t$ denote the average with respect to the Gibbs measure
\begin{equation}\label{eq:gibbs-smartpath}
G_t(\sigma, \alpha) \sim v_\alpha \exp H_{M,t}(\alpha, \sigma).
\end{equation}
We can now compute $\varphi'(t)$ using integration by parts. We first note that the covariance
	\[
	\frac{1}{M_s} \E \frac{\partial H_{M,t} (\sigma^1, \alpha^1)}{\partial_t} H_{M,t} (\sigma^2, \alpha^2) = \frac{1}{M_s} \sum_{i \leq M} 
	\sum_{s \in \sS} \Delta^2_{s, s(i)} \rho^s \Bigl(Q^s_{\alpha^1 \wedge \alpha^2}(\sigma_i^1, \sigma_i^2) - (\tilde Q^s_{\alpha^1 \wedge \alpha^2}(\sigma_i^1, \sigma_i^2) \Bigr).
	\]
The term on the right is bounded in absolute value by $\norm{\Delta} \cdot \abs{\sS}\max_{s\in\sS}
\norm{Q_{\alpha^1\wedge\alpha^2}^s-\tilde Q_{\alpha^1\wedge\alpha^2}^s}_1$. Recalling that the marginals of $G_t(\alpha)$ \eqref{eq:gibbs-smartpath} on $\N^r$ has the same distribution as the weights of $v_\alpha$ \cite[Theorem 4.4]{P3}, a standard Gaussian integration by parts argument will show 
\begin{align*}
\abs{\varphi'(t)} &\leq \E \left\langle \norm{\Delta}\cdot\abs{\sS}\max_{s\in\sS}
\norm{Q_{\alpha^1\wedge\alpha^2}^s-\tilde Q_{\alpha^1\wedge\alpha^2}^s}_1\right\rangle_t\\
 &\leq \|\Delta\||\sS| \sum_{0 \leq p \leq r} \max_{s \in \sS} \| Q^s_{\alpha^1 \wedge \alpha^2} - \tilde Q^s_{\alpha^1 \wedge \alpha^2} \|_1 \E \sum_{\alpha^1 \wedge \alpha^2 = p}  v_{\alpha^1} v_{\alpha^2}\\
&= \|\Delta\||\sS| \int_{0}^{1} \max_{s \in \sS} \| \pi^s(x) - \tilde\pi^s(x) \|_1 \, \de x.
\end{align*}
Integrating this inequality yields the result.
\end{proof}

The following argument is well known and follows by a direct computation using properties of
Ruelle Probability Cascades, see \cite{P2,P3}. 
\begin{lem}\label{lem:ctyofy} The functional $f^Y_{M}(\bm \pi)$ is Lipschitz in $\bm \pi$,
\[
\abs{f^Y_{M}(\bm \pi) - f^Y_{M}( \bm{\tilde \pi})} \leq L \int_{0}^{1} \max_{s \in \sS} \| \pi^s(x) - \tilde\pi^s(x) \|_1 \, \de x.
\]
\end{lem}
We finally observe here that discrete paths are dense in $\pi$.
\begin{lem}\label{thm:discretize} 
For any path $\bm \pi \in \bm \Pi$ and $\epsilon >0$, there exists a finite sequence of points $(x_p)_{p=1}^r$ such that the discrete path
\[
\bm \pi^* (x) := \bm \pi(x_p) \text{ for } x_{p-1} \leq x \leq x_p 
\]
satisfies
\[
\Delta(\bm \pi, \bm  {\pi} ^*) < \epsilon.
\]
\end{lem}

We now state a general continuity theorem regarding functionals of this form. Such results are completely standard, see, e.g., \cite{P2,P3}. Let $S \subset [\kappa]^M$, $(w_\alpha)_{\alpha \in \mathscr{A}}$ be the weights (possibly random) of a probability density distribution on a countable set $\mathscr{A}$, and 
$$R_{\mathscr A} = (R^s_{\alpha^1, \alpha^2})_{s \in \sS, \alpha^1, \alpha^2 \in \mathscr{A}},$$ 
where each $R^s_{\alpha^1, \alpha^2}$ is a $\kappa \times \kappa$ matrix. We think of this array as fixed and non-random. It will be the values that some abstract overlap structure can take. Finally, let 
$(I_s)$ be partition of $[M]$ into species.

We define the functionals
\begin{equation}
f_{1,M} = \frac{1}{M} \E \log \sum_{\alpha \in \mathscr{A}} w_\alpha \sum_{\epsilon \in S} \exp \sum_{i \leq M} Z_i^\alpha(\epsilon_i)
\end{equation}
and
\begin{equation}
f_{2,M} = \frac{1}{M} \E \log \sum_{\alpha \in \mathscr{A}} w_\alpha \exp \sqrt{M} Y^\alpha.
\end{equation}
Here, $Z^\alpha_i(\sigma)$ is a centered Gaussian process is such that, if $i\in I_s\subset [M] $, the covariance structure of the Gaussian vector $Z^\alpha_i = (Z^\alpha_i(k))_{k \leq \kappa}$ is given by
\begin{equation}
\Cov (Z_i^{\alpha^1}, Z_i^{\alpha^2}) = C^{s}_Z((R^s_{\alpha^1, \alpha^2} )_{s \in \sS}).
\end{equation}
for some $C_z^s$, a continuous function of the overlaps $R^s_{\alpha^1, \alpha^2}$. 
Similarly, $Y$ is a centered Gaussian process with covariance given by
\begin{equation}
\Cov (Y^{\alpha^1}, Y^{\alpha^2}) = C_Y((R^s_{\alpha^1, \alpha^2} )_{s \in \sS}).
\end{equation}
The following result is standard and follows from basic properties of Gaussian processes, and the fact that 
log-sums of exponentials have at most linear growth at infinity. Let $(\alpha(\ell))_{\ell\geq1}$ be i.i.d. drawn from 
$\mathscr{A}$ with law $w_\alpha$, and denote
\[
 R^n =\bigl( R^s_{\alpha(\ell), \alpha(\ell')} \bigr)_{ \ell, \ell' \in [n], s \in \sS}.
\] 
We have the following continuity property
\begin{lem}\label{thm:approximation}
For any $\epsilon > 0$, there are continuous bounded functions $g^Z_{\epsilon}$ and $g^Y_\epsilon$ 
such that
\begin{equation}
| f_{1,M}  - \E g^Z_\epsilon (R^n)| \leq \epsilon, \qquad | f_{2,M} - \E g^Y_\epsilon (R^n)| \leq \epsilon.
\end{equation}
These functions depend at most on $M,S,C_z, C_y$, and $\eps$.
\end{lem}
\section{Lower Bound via an Aizenman-Sims-Starr Scheme}\label{sec:Aizenman-LB}
For fixed $d \in \sD$ and a sequence of realizable proportions $d_N \in \sD_N$ \eqref{eq:feasible-constraint-space} converging to $d$, we prove the matching constrained lower bound, 
\begin{equation}\label{eq:lowerbound}
\liminf_{N \to \infty} \frac{1}{N} \E \log Z_N(d_N) \geq \inf_{x, r,(\lambda^s, Q^s)_{s \in \sS}} \sP\bigl(r, x, d, (\lambda^s, Q^s)_{s \in \sS}\bigr).
\end{equation}
By part (1) of \prettyref{thm:MSPOTTSGG}, computing the lower bound of the free energy with respect to 
\begin{equation}\label{eq:perthamiltonian}
H^{pert}_N(\sigma) = H_N(\sigma) + s_N h_N(\sigma)
\end{equation}
introduced in Section \ref{sec:perturbation} and the corresponding constrained partition function
\begin{equation}\label{eq:perturbed-partition-function}
Z_N(d_N) = \sum_{\sigma \in \Sigma_N(d_N)} \exp\bigl(H^{pert}_N(\sigma) \bigr).
\end{equation}
is equivalent to computing the lower bound in \eqref{eq:lowerbound}. We will continue to work with the perturbed Hamiltonian \eqref{eq:perthamiltonian} throughout the remainder of this section.

For $M \geq 1$, our starting point is the following inequality,
\begin{equation}\label{eq:ASSSCHEME}
\liminf_{N \to \infty}\frac{1}{N} \E\log Z_N(d_N) \geq \liminf_{N \to \infty} \frac{1}{M} \bigl( \E \log Z_{N + M} (d_{N + M}) - \E \log Z_N(d_N) \bigr).
\end{equation}
We will write the free energy in terms of the Gibbs measure $G_{d_N}^{pert}$ defined in \eqref{eq:Perturbed-Gibbs-Measure} using the cavity method via an Aizenman-Sims-Starr scheme \cite{AS2}. To this end, let us denote a configuration $\tilde \rho\in[\kappa]^{N+M}$ by $\tilde \rho =(\epsilon, \sigma)$, where $\epsilon = (\epsilon_1, \dots, \epsilon_M) \in [\kappa]^M$ are called the cavity coordinates and $\sigma = (\sigma_{M + 1}, \dots, \sigma_{N + M}) \in [\kappa]^N$ are called the bulk coordinates. We define $\mathcal{M}_s$, $\mathcal{N}_s$ to be the subset of the respective cavity and bulk coordinates that belong to species $s$. Let $M_s$, $N_s$ be the cardinality of $\mathcal{M}_s$ and $\mathcal{N}_s$. 

We control the rate of convergence of $d_N$ so that for some constant $L_\kappa$,
\begin{equation}\label{eq:convergence-proportions}
| d^s_{N,k} - d_k^s | \leq \frac{L_\kappa}{N_s}~\text{for all $k \leq \kappa$, $s \in \sS$} \text{ and $d^s_{N,k} = 0$ if $d_k^s = 0$}.
\end{equation}
A generalized version of \cite[Lemma 11]{P2} will allow us to split constrained configuration space into a product set of the bulk coordinates and the species wise cavity coordinates along a subsequence.

We first introduce some more notation. Let $\mathcal{A} \subset N + M$ and $\mathcal{A}_s$ is the subset of $\mathcal{A}$ in species $s$. If $A := |\mathcal{A}|$ is the cardinality of the set $\mathcal{A}$, we define 
\begin{equation}\label{def:subset-proportions}
\Sigma_{A}(d) := \Bigl \{ (\sigma_i)_{i \in A} \mathrel{}\Bigm{|} \mathrel{} \sum_{i \in \mathcal{A}_s} \1(\sigma_i = k) = A_s d^s_{k}, \forall k \leq \kappa, s \in \sS \Bigr\}
\end{equation}
to be the configurations of spins in $\mathcal{A}$ that satisfy the constraint $d$. If $\mathcal{A} = N + M$, then \eqref{def:subset-proportions} coincides with \eqref{eq:constrainedspace}. The definition in \eqref{def:subset-proportions} naturally implies
\begin{equation}
\Sigma_A(d) = \prod_{s \in \sS} \Sigma_{A_s}(d^s) \label{lem16temp}.
\end{equation}
If we use the sets $\mathcal{N}$ and $\mathcal{M}$ in place of $\mathcal{A}$ in \eqref{lem16temp}, the observations means we can break the cavity and bulk coordinates into a product set over species. We will now show that the entire constrained system $\Sigma_{N+ M}(d)$ contains a product set over the bulk and cavity coordinates.

\begin{lem}
	For every $M > 0$, there exists a constraint $\delta_{M} \in \sD_{M}$ such that 
	\begin{equation}
	|\delta^s_{M,k} - d_k^s| \leq \frac{2 L_\kappa}{M_s} \text{ for all $k \leq \kappa$, $s \in \sS$},
	\end{equation}
	and we can find a subsequence of $N$ such that
	\begin{equation}\label{eq:Separate_into_product_sets}
	\Sigma_{N + M}(d_{N + M}) \supseteq \Sigma_{N}(d_{N}) \times \Sigma_{M} (\delta_{M}).
	\end{equation}
\end{lem}
\begin{proof}
	We first fix $s \in \sS$ and apply \cite[Lemma 11]{P2} to the subset of spins in $\Sigma_{N + M}$ belonging to species $s$. There exists a sequence $\delta^s_{M} \in \sD_{M}$ such that $|\delta^s_{M,k} - d^s_k| \leq \frac{2 L_\kappa}{M_s}$ for all $k \leq \kappa$, and 
	\begin{equation}\label{eq:additive}
	N_s d^s_N + M_s \delta^s_{M} = (N_s + M_s) d^s_{N + M}
	\end{equation}
	 for infinitely many $N_s$. Therefore, we can find a subsequence of $N$ such that that associated $N_s$ satisfies \eqref{eq:additive} and 
	\[
	\Sigma_{N_s + M_s} (d_{N + M}) \supseteq \Sigma_{N_s}(d_N) \times \Sigma_{M_s}(\delta_{M}).
	\]
	Repeating the argument over each species and extracting a further subsequence each iteration, we can conclude
	\[
	\prod_{s \in \sS} \Sigma_{N_s + M_s} (d_{N + M}) \supseteq \prod_{s \in \sS} \Sigma_{N_s}(d^s_N) \times \prod_{s \in \sS} \Sigma_{M_s}(\delta^s_M).
	\]
	Writing the product sets using the observation in \eqref{lem16temp} completes the proof.
\end{proof}
With this observation, we can restrict the first sum in \eqref{eq:ASSSCHEME} to the product set \prettyref{eq:Separate_into_product_sets}, so that \eqref{eq:ASSSCHEME} is bounded below by
\begin{equation}\label{eq:LBstep1}
\frac{1}{M} \bigg( \E \log \sum_{\sigma \in \Sigma_N(d_N)} \sum_{\epsilon \in \Sigma_{M}(\delta_{M})} \exp \bigl(H^{pert}_{N + M} (\sigma, \epsilon) \bigr) - \E \log \sum_{\sigma \in \Sigma_N(d_N)} \exp \bigl(H^{pert}_N(\sigma) \bigr) \bigg).
\end{equation}

We will now separate the unperturbed portion of the Hamiltonians into its cavity fields,
\begin{align}
H_{N + M} (\epsilon,\sigma) &= H_N'(\sigma) + \sum_{i \leq M} Z_{N,i}^\sigma(\epsilon_i) + r (\epsilon), \\
H_{N} (\sigma) &= H_N'(\sigma) + \sqrt{M} Y_N^\sigma.
\end{align}
These cavity fields are the same as those appearing in \cite[Equation (107) and (109)]{P2}, except with different covariance structure because of the inhomogeneity. In our case, these fields are independent Gaussian processes with covariance 
\begin{alignat}{2}
\E H'_N(\sigma^\ell)H'_N(\sigma^{\ell'}) &= \frac{N^2}{N + M} \sum_{s,t \in \sS} \Delta^2_{s,t} \rho_N^s \rho_N^t ( R^s_{\ell,\ell'},  R^t_{\ell,\ell'}),\\
\E Z_{N,i}^{\sigma^\ell}(\epsilon) Z_{N,j}^{\sigma^{\ell'}}(\epsilon') &= 2 \delta_{\{i = j\}} \sum_{s \in \sS} \Delta^2_{s(i),s} \rho_N^s R^s_{\ell,\ell'}(\epsilon, \epsilon') + O(N^{-1}) , \\
\E Y_N^{\sigma^\ell}Y_N^{\sigma^{\ell'}} &= \sum_{s,t \in \sS} \Delta^2_{s,t} \rho_N^s \rho_N^t ( R^s_{\ell,\ell'},  R^t_{\ell,\ell'}) + O(N^{-1}).
\end{alignat}
Here $\rho_N^s := \frac{\sum_{M + 1 = 1}^{N + M} \1(i \in I_s)}{N}$ is the proportion of the bulk coordinates that belong to species $s$. This distinction is not critical because this proportion converges to the same $\rho^s$ defined in \eqref{eq:proportion-def}. The $r(\epsilon)$ term is $O(N^{-1})$ and can be omitted without affecting \eqref{eq:lowerbound}. 

By a standard interpolation argument, see for example \cite[Theorem 3.6]{P3}, if $(z^\sigma_{N,i}(\epsilon))_{i \leq M}$ and $y_N^\sigma$ are centered Gaussian processes with covariances
\begin{align}
	\E z_{N,i}^{\sigma^\ell}(\epsilon) z_{N,j}^{\sigma^{\ell'}}(\epsilon') &= 2 \delta_{\{i = j\}} \sum_{s \in \sS} \Delta^2_{s(i),s} \rho^s R^s_{\ell,\ell'}(\epsilon, \epsilon'),\label{eq:cov-gibbs-z}\\
	\E y_N^{\sigma^\ell}y_N^{\sigma^{\ell'}} &= \sum_{s,t \in \sS} \Delta^2_{s,t} \rho^s \rho^t ( R^s_{\ell,\ell'},  R^t_{\ell,\ell'})\label{eq:cov-gibbs-y},
\end{align}
then \eqref{eq:LBstep1} is bounded below by
\begin{equation}\label{eq:LBstep2}
	\frac{1}{M}\E \log \bigg \langle \sum_{\epsilon \in \Sigma_{M}(\delta)} \exp{\sum_{i \leq M} z^\sigma_{N,i} (\epsilon)} \bigg \rangle_{G'_N} -\frac{1}{M} \E \log \bigg\langle \exp{\sqrt{M} y^\sigma_N} \bigg\rangle_{G'_N} + o(1).
\end{equation}
Here $\langle \cdot \rangle_{G_N'}$ is the average in $\sigma$ with respect to the perturbed Gibbs measure \eqref{eq:Perturbed-Gibbs-Measure}.
By \prettyref{thm:approximation}, the functionals appearing in \eqref{eq:LBstep2} are continuous functionals of the distribution of the overlap matrix array 
\begin{equation}
 R^N = \bigl(R^s_{\ell, \ell'}\bigr)_{\ell, \ell' \geq 1, s \in \sS}
\end{equation}
of i.i.d. draws from the perturbed Gibbs measure $G_{d_N}^{pert}$. We will now relate \eqref{eq:LBstep2} to the Ruelle Probability Cascades allowing us to compute its value explicitly.

In the following, we take a subsequence along which the limit inferior of \prettyref{eq:LBstep2} is achieved. Since $\cR$ is compact we may take a subsequential weak limit of $R^N$ along our minimizing subsequence, which we denote by $R^\infty$. For ease of notation, we will continue to denote this subsequence with $N$. By the choice of the perturbation Hamiltonian, \prettyref{thm:MSPOTTSGG}, we have the limiting array $R^\infty$ satisfies equation \eqref{eq:MSPGGI}. By the characterization theorem, \prettyref{thm:IP-characterization}, we see the order parameter for this system, that is the quantity which determines the law of the system, will be the law of $\tr (\bar R^\infty_{12})$. 

With this in mind, we make the following approximation. The array, $(\tr(\bar R^\infty_{\ell,\ell'}))_{\ell,\ell' \geq 1}$, by definition of $\cR$, is a Gram-De Finetti array. Furthermore, taking $\nu^s_n = \sqrt{\rho^s} e_n$, $p_s = 1$, and $\phi := \sum_{s \in \sS} \sum_{i \leq \kappa} (\nu^s_i, R^{s,\infty}_{\ell,\ell'} \nu^s_i) $ 
in \prettyref{eq:MSPGGI} the array $\bigl(\tr(\bar R^\infty_{\ell,\ell'}) \bigr)_{\ell,\ell' \geq 1}$ also satisfies the classical Ghirlanda-Guerra identities. By \cite[Thereom 2.13]{P3}, the law of $(\tr(\bar R^\infty_{\ell,\ell'}))_{\ell, \ell' \geq 1}$ is uniquely determined by $\zeta$, the law of $\tr \bigl(\bar R^\infty_{1,2} \bigr)$.
Let $\zeta^n\to\zeta$ weakly in $\Pr[0,1]$, such that $\zeta^n$ consists of a finite number of atoms. This yields sequences 
\begin{align}
x^n_{-1}= 0 &< x^n_0 < \ldots < x^n_{r-1} < x^n_r = 1,
\label{eq:RPC-sequences}
\\
0 &=q^n_0 \, <  \ldots  < q^n_{r-1} < q^n_r = 1,
\nonumber
\end{align}
such that $\zeta^n([0,q_p^n])=x_p^n$. By \cite[Theorem 2.17]{P3}, if $\tr \bigl(\bar R^n_{1,2} \bigr)$ has distribution $\zeta^n$, then the approximating array of traces $ \bigl( \tr (\bar R^n_{1,2}) \bigr)_{\ell, \ell' \geq 1}$ converge weakly to $(\tr(\bar R^\infty_{\ell,\ell'}))_{\ell, \ell' \geq 1}$.

Let $v_\alpha$ be the weights of the Ruelle probability cascades associated with the sequence $(x^n)$ in \eqref{eq:RPC-sequences}. If $(a_\ell)_{\ell \geq 1}$ are i.i.d samples from $\N^r$ according to $v_\alpha$ then it is well known \cite[Section 3.6]{P3} that
$T^n_{\ell,\ell'} = q^n_{\alpha^\ell \wedge \alpha^{\ell'}}$ will be close in distribution to $ \bigl(\tr( \bar R^\infty_{\ell,\ell'}) \bigr)_{\ell,\ell' \geq 1}$. More precisely,
\begin{equation}\label{eq:approx-trace}
\bigl(T^n_{\ell,\ell'} \bigr)_{\ell,\ell' \geq 1} \stackrel{d}{\to} \bigl( \tr( \bar R^\infty_{\ell,\ell'}) \bigr)_{\ell,\ell' \geq 1} .
\end{equation}

Since the limiting array $R^\infty$ is IP-Invariant \eqref{eq:MSPGGI}, by \prettyref{thm:IP-characterization} we can find a family of Lipschitz functions $(\Phi_s)_{s \in \sS}$ on $[0,1]$ such that
\[
 R^\infty = (R_{\ell,\ell'}^{s,\infty})_{\ell,\ell' \geq 1, s \in \sS} =  \bigl( \Phi_s \bigl( \tr( \bar R^\infty_{\ell,\ell'}) \bigr) \bigr)_{\ell,\ell' \geq 1, s \in \sS}.
\]
Since $\Phi_s$ is Lipshitz, \eqref{eq:approx-trace} implies the $\kappa \times \kappa$ matrices $Q^{s,n}_{\ell,\ell'} := \Phi_s (T^n_{\ell,\ell'} )$ will be close in distribution to $R^{s,\infty}_{\ell,\ell'}$. If we let $ Q^n := \bigl(Q^{s,n}_{\ell,\ell'} \bigr)_{\ell,\ell' \geq 1, s \in \sS}$ denote the approximating array generated from \iid samples under Ruelle probability cascades corresponding to order parameter $\zeta^n$ then \eqref{eq:approx-trace} also implies
\begin{equation}\label{eq:convergence-Q}
 Q^n  \stackrel{d}{\to}  R^\infty.
\end{equation}

For $n$ sufficiently large, we will bound \eqref{eq:LBstep2} arbitrarily closely with functionals of the infinite array $Q^n$, which we will now show. Let $\big( (Z_{i,n}^{\alpha}(k))_{k\leq\kappa} \big)_{i \leq M}$ and $Y_n^\alpha$ be 
centered Gaussian processes with covariances
\begin{align}
\E Z_{i,n}^{\alpha}Z_{j,n}^{\beta} &= 2 \delta_{\{i = j\}} \sum_{t \in \sS}\Delta_{s(i),t}^{2}\rho^{t}Q_{\abs{\alpha\wedge\beta}}^{t,n}, \label{eq:cov-RPC-z}\\
\E Y_n^{\alpha}Y_n^{\beta} &=\sum_{s,t \in\sS}\Delta_{s,t}^{2}\rho^{s}\rho^{t} 
\left(Q_{\abs{\alpha\wedge\beta}}^{s,n},Q_{\abs{\alpha\wedge\beta}}^{t,n}\right) \label{eq:cov-RPC-y}.
\end{align}
Notice that the covariance structure in \eqref{eq:cov-RPC-z} and \eqref{eq:cov-RPC-y} depend on the overlap array in exactly the same way as \eqref{eq:cov-gibbs-z} and \eqref{eq:cov-gibbs-y}. A direct application of the continuity in \prettyref{thm:approximation} and the convergence in distribution of $ Q^n$ in \eqref{eq:convergence-Q} will imply the following result.

\begin{lem}\label{lem:RPCLWBD} 
	For every $\tilde \epsilon > 0$, there is an $n > 0$ such that 
	\begin{align}
	& \liminf_{N \to \infty} \frac{1}{M}\E \log \bigg \langle \sum_{\epsilon \in \Sigma_{M}(\delta)} 
	\exp{\sum_{i \leq M} z^\sigma_{N,i} (\epsilon)} \bigg \rangle_{G'_N} 
	-\frac{1}{M} \E \log \bigg\langle \exp{\sqrt{M} y^\sigma_N} \bigg\rangle_{G'_N} \notag\\
	&\qquad\geq \frac{1}{M} \E \log \sum_{\alpha \in \N^r} v_\alpha 
	\sum_{\epsilon \in \Sigma_M(\delta)} \exp \sum_{i \leq M} Z_{i,n}^\alpha (\epsilon_i) 
	- \frac{1}{M} \E \log \sum_{\alpha \in \N^r} v_\alpha  \exp \sqrt{M} Y_n^\alpha- \tilde \epsilon. \label{eq:temp4}
	\end{align}
\end{lem}

\begin{proof} 
The covariance structure of \eqref{eq:cov-RPC-z} and \eqref{eq:cov-RPC-y} are identical to \eqref{eq:cov-gibbs-z} and \eqref{eq:cov-gibbs-y}. By \prettyref{thm:approximation} the functionals appearing in \eqref{eq:temp4} can be approximated by the same bounded continuous functions $\E g_\epsilon^Z( R)$ and $\E g_\epsilon^{Y}( R)$ of the overlap arrays. Since the distribution of the arrays $ R^N$ and $ Q^n$ both converge weakly to $ R^\infty$, by first taking a subsequence $N$ along which $ R^N$ converges to $ R^\infty$ and then approximating with $ Q^n$ we have for any $\tilde \epsilon > 0$, 
\[
|\E g_\epsilon^Z( R^\infty) - \E g_\epsilon^{Y}( R^\infty) - \E g_\epsilon^Z( Q^n) + \E g_\epsilon^{Y}( Q^n)| \leq \tilde \epsilon.
\]
by choosing $n$ sufficiently large. Applying the triangle inequality will complete the proof. 
\end{proof}
Since $\Sigma_M(\delta_M) = \prod_{s \in \sS} \Sigma_{M_s} (\delta_{M_s})$, 
we express the first term in \eqref{eq:temp4} as a weighted average. By the properties of the Ruelle probability cascades (see  \cite[Theorem 2.9]{P3}), 
\begin{align}
\frac{1}{M} \E \log \sum_{\alpha \in \N^r} v_\alpha 
\sum_{\epsilon \in \Sigma_M(\delta)} \exp \sum_{i \leq M} Z_{i,n}^\alpha (\epsilon_i) &= \frac{1}{M}  \E \log \sum_{\alpha \in \N^r} v_\alpha 
\prod_{s \in \sS} \sum_{\epsilon \in \Sigma_{M_s}(\delta_{M_s})} \exp \sum_{i \in M_s} Z_{i,n}^\alpha (\epsilon_i) \nonumber\\
&= \sum_{s \in \sS} \rho^s_M  \frac{1}{M_s} \E \log \sum_{\alpha \in \N^r} v_\alpha 
\sum_{\epsilon \in \Sigma_{M_s}(\delta_{M_s})} \exp \sum_{i \in M_s} Z_{i,n}^\alpha (\epsilon_i). \nonumber
\end{align}
Every sequence $(x^n_p)_{p=0}^r$ and $(Q^{n,s}_p)_{p=0}^r$ defines a discrete path $\bigl(\pi_n^s(x) \bigr)_{s \in \sS} \in \Pi$. Using the notation of the functionals $f^Z_{M_s}(\Sigma_{M_s}(\delta_{M_s}), s; \bm \pi_n)$ and $f^Y_M(\bm \pi_n)$ defined on \eqref{eq:func-pi-z} and \eqref{eq:func-pi-y}, we have shown that
\[
\liminf_{N \to \infty}\frac{1}{N} \E \log Z_N(d_N) \geq \Big( \sum_{s \in \sS} \rho_M^s f^Z_{M_s}(\Sigma_{M_s}(\delta_{M_s}), s; \bm \pi_n) - f^Y_M (\bm \pi_n) \Big) - \epsilon.
\]
\prettyref{lem:ctyofz} and \prettyref{lem:ctyofy} imply the functionals $f^Z_{M_s}(\Sigma_{M_s}(\delta_{M_s}), s;\bm \pi_n)$ and $f^Y_M (\bm \pi_n)$ are Lipschitz. Sending $\epsilon \to 0$ and noticing that the paths $\bm \pi_n \to \bm \pi_\infty$ in $\Delta$, we have shown
\begin{equation}\label{eq:temp2}
\liminf_{N \to \infty}\frac{1}{N} \E \log Z_N(d_N) \geq \Big( \sum_{s \in \sS} \rho_M^s f^Z_{M_s}(\Sigma_{M_s}(\delta_{M_s}),s;\bm \pi_\infty) - f^Y_M (\bm \pi_\infty) \Big) , \text{ for any $M > 0$}.
\end{equation}

All that remains is to remove the dependence on $\Sigma_M(\delta_M)$. This will be a direct application of the decoupling proved in \prettyref{thm:discretize}.

\begin{lem} There exists a path $\bm \pi^* \in \Pi$ such that
	\[
	\lim_{M \to \infty} \Big( \sum_{s \in \sS} \rho_M^s f^Z_{M_s}(s,\bm \pi^*) - f^Y_M (\bm \pi^*) \Big) 
	\geq \sP(r, x, d, (\lambda^s, Q^s)_{s \in \sS}) - \epsilon.
	\]
\end{lem}

\begin{proof}
	We now make the paths $\bm \pi^M_\infty$ dependence on $M$ explicit. Since $(\bm \pi^M_\infty)_{M \geq 1}$ is a countable collection of bounded monotone paths of $\kappa \times \kappa$ matrices, there exists a subsequence in $M$ such that $\bm \pi^M_\infty \to \bm \pi^*$. Given $\epsilon > 0$, by \prettyref{thm:discretize} we can find a discrete path $\bm \pi_\epsilon$ such that $\Delta(\bm \pi_\epsilon, \bm \pi^*) < C\epsilon$, where $C$ is the maximum Lipshitz constant over all $f^Z_{M_s}(s)$. Since $\bm \pi^\epsilon$ is discrete, applying \prettyref{thm:decouple} to the summation appearing in \eqref{eq:temp2} shows
	\[
	\lim_{M \to \infty}  \sum_{s \in \sS} \rho_M^s f^Z_{M_s}(s, \bm \pi^*) \geq \lim_{M \to \infty} \sum_{s \in \sS} \rho_M^s f^Z_M(\bm \pi_{\epsilon}) - C\epsilon \geq \inf_{\lambda} \sum_{s \in \sS} \rho^s \bigg( -\sum_{k \leq \kappa}\lambda^s_k d_k^s +  X_0^s(\bm \pi^\epsilon) \bigg) - C\epsilon
	\]
	since $\rho^s_M \to \rho^s$, $\delta^s_{M,k} \to d_k^s$ and $f_M^z$ is Lipschitz. 
	
	The term $f^Y_M(\bm \pi_\infty)$ is actually independent of $M$, and a similar computation using the properties of Ruelle Probability Cascades like in \eqref{eq:upbdyterm} shows
	\[
	\lim_{M \to \infty}   f^Y_M(\bm \pi^*) = \frac{1}{2} \sum_{\ell = 0}^{r-1} x_{\ell} \sum_{s,t \in \sS} \Delta^2_{st} \rho^s \rho^t \left( (Q^s_{\ell+1},Q^t_{\ell+1}) - (Q^s_{\ell},Q^t_{\ell}) \right).
	\]
	
	Combining the computations of the two terms above, we now take $\epsilon \to 0$ and get
	\begin{equation}\label{eq:lowerbd-end}
	\liminf_{N \to \infty} \frac{1}{N}\E \log Z_N(d_N) \geq \sP\bigl(d, \lambda, \bm \pi^*\bigr)
	\end{equation}
	since $\pi_\epsilon \to \pi^*$ as $\epsilon \to 0$. 
\end{proof}
Since the path $\bm \pi^*$ can be described as the limit of the discrete approximating sequences \eqref{eq:RPC6} and \eqref{eq:RPC5}, we have shown
\[
\liminf_{N \to \infty} \frac{1}{N}\E \log Z_N(d_N) \geq \inf_{x, r,(\lambda^s, Q^s)_{s \in \sS}} \sP\bigl(r, x, d, (\lambda^s, Q^s)_{s \in \sS}\bigr).
\]

\section{Proofs of Main Theorem}

\begin{proof}[Proof of \prettyref{thm:ParisiUnConstrained}] 
	We begin by proving part (1) of the Theorem.
			
	We start with the lower bound of $F_N \bigl(\Sigma_N^{\epsilon_N}(d) \bigr)$. For any $d \in \sD$ let us choose $\epsilon_N = \frac{L_\kappa}{N}$. We choose a sequence of $d_N \in \sD_N$ such that $\|d_N - d\|_\infty = \epsilon_N$ to satisfy the condition in \eqref{eq:convergence-proportions}. Since $\Sigma^\epsilon_N(d) \supseteq \Sigma_N(d_N)$, the lower bound \eqref{eq:lowerbound} computed in \prettyref{sec:Aizenman-LB} implies
	\begin{equation}\label{eq:mainthm-cons-lb}
	\liminf_{N \to \infty} \frac{1}{N} \E \log Z_N^{\epsilon_N}(d) \geq \liminf_{N \to \infty} \frac{1}{N} \E \log Z_N(d_N) \geq \inf_{x, r,(\lambda^s, Q^s)_{s \in \sS}}  \sP\bigl(r, x, d, (\lambda^s, Q^s)_{s \in \sS} \bigr).
	\end{equation}
	We now obtain the matching upper bound of $F_N \bigl(\Sigma_N^{\epsilon_N}(d) \bigr)$. Since $\epsilon_N \to 0$, the $O(\epsilon)$ term in the upper bound \eqref{eq:guerra-ub-thick} computed in \prettyref{sec:Guerra-UB} vanishes yielding
	\begin{equation}\label{eq:mainthm-cons-ub}
	\limsup_{N \to \infty} \frac{1}{N} \E \log Z_N^{\epsilon_N}(d) \leq \inf_{x, r,(\lambda^s, Q^s)_{s \in \sS}} \sP\bigl(r, x, d, (\lambda^s, Q^s)_{s \in \sS} \bigr).
	\end{equation}
	Combining the inequalities in \eqref{eq:mainthm-cons-lb} and \eqref{eq:mainthm-cons-ub}, we arrive at the formula for the constrained free energy
	\begin{equation}\label{eq:mainthm-cons-functional}
	\lim_{N \to \infty} \frac{1}{N} \E \log Z_N(d) = \inf_{x, r,(\lambda^s, Q^s)_{s \in \sS}} \sP\bigl(r, x, d, (\lambda^s, Q^s)_{s \in \sS} \bigr).
	\end{equation}
	
	Part (2) of the Theorem is a direct consequence of part (1). By classical Gaussian concentration inequalities, observe that the limit of the free energy $F_N(\Sigma_N)$ is asymptotically given by the supremum of $F_N\bigl(\Sigma_N(d)\bigr)$ over $d \in \sD$. By taking the supremum over $d \in \sD$ in \eqref{eq:mainthm-cons-functional},  we arrive at the formula for the free energy
	\begin{equation}
	\lim_{N \to \infty} \frac{1}{N} \E \log Z_N = \sup_{d \in \sD} ~\inf_{x, r,(\lambda^s, Q^s)_{s \in \sS}} \sP\bigl(r, x, d, (\lambda^s, Q^s)_{s \in \sS} \bigr).
	\end{equation}
\end{proof}

\begin{proof}[Proof of Corollary \ref{cor:ground_state}] 
We include a proof of the Corollary for the sake of completeness. By H\"{o}lder's inequality, 
$\frac{F_N^{\beta}(\Sigma_N^{\ve_N}(d))}{\beta}$ is increasing in $\beta$ and therefore 
\begin{align}
\lim_{\beta \to \infty} \frac{1}{\beta}\inf_{x, r,(\lambda^s, Q^s)_{s \in \sS}} \sP_\beta\bigl(r, x, d,(\lambda^s, Q^s)_{s \in \sS} \bigr)
\end{align}
exists. Moreover, we have, 
\begin{align}
\frac{1}{N} \E \Big[ \max_{\Sigma_{N}^{\ve_N}(d)} H_N(\sigma) \Big] \leq \frac{F_N^{\beta}(\Sigma_N^{\ve_N}(d))}{\beta} \leq \frac{\log q}{\beta} + \frac{1}{N} \E \Big[ \max_{\Sigma_{N}^{\ve_N}(d)} H_N(\sigma) \Big]. \nonumber 
\end{align}
Taking the limit as $N\to \infty$ and then as $\beta \to \infty$ completes the proof. 
\end{proof}
\section{Proofs regarding Cuts Problems}
\label{sec:cut_proofs}
In this section, we prove Theorem \ref{thm:discretization}. We start with a proof of Lemma \ref{lemma:discretization}. 
\begin{proof}[Proof of Lemma \ref{lemma:discretization}]
The proof of this lemma proceeds by a direct coupling argument. We will denote the adjacency matrix corresponding to $\tilde{G}_N$ as $\tilde{A} = ( \tilde{A}_{i,j})$. We will realize the two graphs on the same probability space and couple the adjacency matrices such that 
\begin{align}
\P[ A_{i,j} \neq \tilde{A}_{i,j}] &= | \P[ A_{i,j} =1] - \P[\tilde{A}_{i,j} =1 ] | \nonumber \\
&\leq c N \Big| \int_{[\frac{i-1}{N}, \frac{i}{N}] \times [\frac{j-1}{N},\frac{j}{N}]} K(x,y) \, \de x \de y -  \int_{[\frac{i-1}{N}, \frac{i}{N}] \times [\frac{j-1}{N},\frac{j}{N}]} K_1(x,y) \, \de x \de y \Big|. \nonumber 
\end{align}
We note that for any two graphs $G_N$ and $\tilde{G}_N$, we have, 
\begin{align}
\Big| \frac{\mqcut(G_N)}{N}- \frac{\mqcut(\tilde{G}_N)}{N}  \Big| \leq \frac{1}{2N} \sum_{i,j = 1}^N |A_{i,j} - \tilde{A}_{i,j}|. \nonumber 
\end{align}
This implies that 
\begin{align}
&\Big| \E \Big[ \frac{\mqcut(G_N)}{N} \Big]- \E\Big[\frac{\mqcut(\tilde{G}_N)}{N} \Big] \Big| \nonumber \\ 
 &\leq \frac{c}{2} \sum_{i,j = 1}^N \Big| \int_{[\frac{i-1}{N}, \frac{i}{N}] \times [\frac{j-1}{N},\frac{j}{N}]} K(x,y) \, \de x \de y -  \int_{[\frac{i-1}{N}, \frac{i}{N}] \times [\frac{j-1}{N},\frac{j}{N}]} K_1(x,y) \, \de x \de y \Big| \nonumber \\
&= \frac{c}{2} \bigl( \| K - K_1 \|_1 + o(1) \bigr). \nonumber 
\end{align}
Thus the proof is complete once we choose $M$ sufficiently large such that $\|K - K_1 \|_1 \leq 1/ c^{1/2 + \delta}$. 

\end{proof}

To complete the proof of Theorem \ref{thm:discretization}, we first introduce the following Gaussian optimization problem. Set $J = (J_{i,j})$ a symmetric matrix such that $\{J_{i,j}: i \leq j \}$ are independent $N(0, \frac{\tilde{K}_N(i,j)}{N})$ random variables. We define 
\begin{align}
\tilde{Z}_N = \frac{1}{2N} \max_{\sigma \in [\kappa]^N} \Big[ \frac{c}{N} \sum_{i,j = 1}^N \tilde{K}_N(i,j) \bone(\sigma_i \neq \sigma_j) + \sqrt{c}\sum_{i,j = 1}^N \frac{J_{i,j}}{\sqrt{N}} \bone(\sigma_i \neq \sigma_j) \Big].  
\end{align}
The following lemma establishes that in the ``large degree" limit, we can study the asymptotic behavior of the $\mqcut$ problem via that of the Gaussian optimization problem $\tilde{Z}_N$. 

\begin{lem}
\label{lemma:comp}
As $N \to \infty$, we have, 
\begin{align}
\E\Big[ \frac{\mqcut(G_N)}{N} \Big] = \E[\tilde{Z}_N] + o(\sqrt{c}). \nonumber 
\end{align}
\end{lem}

\begin{proof}
The lemma follows from a direct application of \cite[Theorem 1.1]{sen2016optimization}. Specifically, we note that 
\begin{align}
\frac{\mqcut(G_N)}{N} = \frac{1}{2N} \max_{\sigma \in [\kappa]^N} \sum_{i,j = 1}^N A_{i,j} \bone(\sigma_i \neq \sigma_j). \nonumber 
\end{align}
The proof follows directly upon an application of the result in \cite{sen2016optimization}. 
\end{proof}

Thus finally, it comes down to the study of $\E[Y_n]$ as $N\to \infty$. Fix a probability distribution on $[\kappa]$, $d^{s} = (d^{s}_1, \ldots, d^{s}_\kappa)$, $s= 1, \ldots, M$. Consider the Hamiltonian 
\begin{align}
H(\sigma) = \sum_{i,j =1}^N \frac{J_{i,j}}{\sqrt{N}} \bone(\sigma_i = \sigma_j),  
\end{align}
where $J = (J_{i,j})$ is as described above. 
For $d = \bigl(d^{1}, \ldots, d^{M} \bigr)$, probability measures on $[\kappa]$, and a sequence $\ve_N$ decaying to zero sufficiently slowly, recall the restricted configuration space \eqref{eq:constrainedspace}
\begin{align}
\Sigma_N^{\ve_N}(d) = \Big\{ \sigma \in [\kappa]^N \mathrel{}\Big|\mathrel{} \sum_{i \in I_s} \frac{\bone(\sigma_i = k )}{N_s} \in [ d^{s}_k -\ve_N, d^{s}_k + \ve_N], s \in [M] \Big\}.  
\end{align}

Recall the restricted ground state energy $\mathcal{P}(d)$, introduced in Corollary \ref{cor:ground_state}. 
 This will allow us to deduce the following lemma.

\begin{lem}
\label{lemma:exp}
\begin{align}
\lim_{N \to \infty} \E[\tilde{Z}_N] = \sup_{d} \Big[ \frac{c}{2} \sum_{s,t = 1}^M \mathbf{K}(s, t) \rho^{s} \rho^{t} (1- \langle d^{s}, d^{t} \rangle) + \frac{\sqrt{c}}{2} \mathcal{P}(d) \Big] + o(\sqrt{c}). \nonumber
\end{align}
\end{lem}

\begin{proof}
We start with the lower bound.
We define 
\begin{align}
\tilde{Z}_N(d) &= \frac{1}{2N} \max_{\sigma \in \Sigma_N^{\ve_N}(d)} \Big[ \frac{c}{N} \sum_{i,j = 1}^N \tilde{K}_N(i,j) \bone(\sigma_i \neq \sigma_j) + \sqrt{c}\sum_{i,j = 1}^N \frac{J_{i,j}}{\sqrt{N}} \bone(\sigma_i \neq \sigma_j)\Big] \nonumber\\
&= \frac{c}{2} \sum_{s,t = 1}^M \mathbf{K}(s,t) \rho^{s} \rho^{t} \bigl(1 - \langle d^{s}, d^{t} \rangle \bigr) + \frac{\sqrt{c}}{2N}\max_{\sigma \in \Sigma_N^{\ve_N}(d)}  \sum_{i,j = 1}^N \frac{J_{i,j}}{\sqrt{N}} \bone(\sigma_i \neq \sigma_j) + o(1). \nonumber 
\end{align}

For fixed probability vectors $d^{1}, \ldots, d^{M}$, we have, 
\begin{align}
\liminf_{N \to \infty} \E[\tilde{Z}_N] &\geq \liminf_{N \to \infty} \E[\tilde{Z}_N(d)]  \nonumber \\
&= \frac{c}{2} \sum_{s,t = 1}^M \mathbf{K}(s,t) \rho^s \rho^t \bigl(1 - \langle d^s, d^t \rangle \bigr) + \frac{\sqrt{c}}{2} \mathcal{P}(d). \nonumber 
\end{align}

We take the supremum over all possible probability vectors $d^{1}, \ldots , d^{M}$ to get the requisite lower bound. To establish the upper bound, we define 
\begin{align}
\tilde{M} =  \sup_{d} \Big[ \frac{c}{2} \sum_{s,t = 1}^M \mathbf{K}(s, t) \rho^s \rho^t \bigl(1- \langle d^{s}, d^{t} \rangle \bigr)+ \frac{\sqrt{c}}{2} \mathcal{P}(d) \Big].  
\end{align}
Therefore, we have,
\begin{align}
\E[\tilde{Z}_N(d)] &\leq \frac{c}{2} \sum_{s,t = 1}^M \mathbf{K}(s,t) \rho^s \rho^t \bigl(1 - \langle d^{s}, d^{t} \rangle \bigr) + \frac{\sqrt{c}}{2} \mathcal{P}(d) + o(1). \nonumber \\
&\leq \tilde{M} + o(1), \nonumber
\end{align}
uniformly over all choices of $d^{1}, \ldots, d^{M}$. We note that empirical distributions within each block may assume only finitely many values and hence, summing over these values, we have, 
\begin{align}
\P[ \tilde{Z}_N > \tilde{M} + t] = \sum_{d^1, \ldots, d^M} \P \bigl[ \tilde{Z}_N(d) > \tilde{M} +t \bigr] . \nonumber 
\end{align}
Now, 
\begin{align*}
\P[\tilde{Z}_N(d)> \tilde{M} +t ] &\leq \P[ \tilde{Z}_N (d) - \E[\tilde{Z}_N(d)] > t] 
 \leq  \exp[- C N t^2], 
 \end{align*}
 where the last inequality follows by Gaussian concentration. Finally, plugging this tail bound, we have, 
 \begin{align}
 \P[\tilde{Z}_N > \tilde{M} + t ] \leq A N^{\kappa \tilde{M}} \exp[-CN t^2]. \nonumber  
 \end{align}
We note that 
 \begin{align}
 \E[\tilde{Z}_N] &\leq \int_{0}^{\infty} \P [ \tilde{Z}_N > x] \, \de x \leq \tilde{M} +  \delta_N + \int_{ \delta_N}^{\infty} \P[ \tilde{Z}_N > \tilde{M} + t] \, \de t \nonumber \\
 &\leq \tilde{M} + \delta_N +  A N^{\kappa \tilde{M}} \frac{\exp[- CN \delta_N^2]}{\sqrt{N}\delta_N} . \nonumber 
 \end{align}
 
 Finally, we choose $\delta_N = C_0 \sqrt{\frac{\log N}{N}}$ for some constant $C_0$ sufficiently large. This establishes that 
 \begin{align}
 \E[\tilde{Z}_N] \leq \tilde{M} + o(1), \nonumber 
 \end{align}
 completing the proof of the upper bound. 
\end{proof}

\begin{proof}[Proof of Theorem \ref{thm:discretization}]
The Theorem follows directly upon combining Lemma \ref{lemma:comp} and Lemma \ref{lemma:exp}.  

\end{proof}

\end{document}